\documentclass[sn-mathphys]{sn-jnl}
\jyear{2022}%

\theoremstyle{thmstyleone}%
\newtheorem{theorem}{Theorem}

\usepackage{amsmath,amsthm}
\newtheorem{lemma}{Lemma}
\numberwithin{equation}{section}

\theoremstyle{thmstyletwo}%

\theoremstyle{thmstylethree}%
\newtheorem{definition}{Definition}%

\begin{document}

\small

\title[Chemical systems with limit cycles]{Chemical systems with limit cycles}

\author[1]{\fnm{Radek} \sur{Erban}}\email{erban@maths.ox.ac.uk}

\author[2]{\fnm{Hye-Won} \sur{Kang}}\email{hwkang@umbc.edu}

\affil[1]{\orgdiv{Mathematical Institute}, 
\orgname{University of Oxford},
\orgaddress{\street{Radcliffe Observatory Quarter, Woodstock Road}, 
\city{Oxford}, \postcode{OX2 6GG}, \country{United Kingdom}}}

\affil[2]{\orgdiv{Department of Mathematics and Statistics},
\orgname{University of Maryland, Baltimore County},
\orgaddress{\street{1000 Hilltop Circle}, \city{Baltimore}, \postcode{21250}, 
\state{Maryland}, \country{USA}}}

\abstract{The dynamics of a chemical reaction network (CRN) is often modelled
under the assumption of mass action kinetics by a system of ordinary differential 
equations (ODEs) with polynomial right-hand sides that describe the time evolution 
of concentrations of chemical species involved. Given an arbitrarily large integer 
$K \in {\mathbb N}$, we show that there exists a CRN such that its ODE model has 
at least $K$ stable limit cycles. Such a CRN can be constructed with reactions of
at most second order provided that the number of chemical species grows linearly 
with $K$. Bounds on the minimal number of chemical species and the minimal number
of chemical reactions are presented for CRNs with $K$~stable limit cycles and at most 
second order or seventh order kinetics. We also show that CRNs with only two 
chemical species can have $K$ stable limit cycles, when the order of 
chemical reactions grows linearly with $K$. 
}


\maketitle

\section{Introduction}\label{sec1}

Chemical reaction networks (CRNs) are often modelled by reaction rate equations, which are systems of first-order, autonomous, ordinary differential equations (ODEs) describing the time evolution of the concentrations 
of chemical species involved. Considering CRNs which are subject to the law of mass action, their reaction rate equations have
polynomials on their right-hand sides~\cite{Yu:2018:MAC,Craciun:2020:RKD}. The mathematical investigation of ODEs with 
polynomial right-hand sides has a long history and includes a number of challenging open mathematical problems, for example, Hilbert's 16$^{\mathrm{th}}$ Problem~\cite{Ilyashenko:2002:CHH}, which asks questions about the number and position of limit cycles of the planar ODE system of the form
\begin{eqnarray}
\frac{\mbox{d}x}{\mbox{d}t}
& = &
f(x,y), 
\label{xode}
\\
\frac{\mbox{d}y}{\mbox{d}t}
& = &
g(x,y), 
\label{yode}
\end{eqnarray}
where $f(x,y)$ and $g(x,y)$ are real polynomials of degree at most $n$. Denoting $H(n)$ the maximum number of limit cycles for the system~(\ref{xode})--(\ref{yode}), neither the value of $H(n)$ (for $n \ge 2$) nor any upper bound on $H(n)$ have yet been found~\cite{Ilyashenko:1991:FTL}. Since a quadratic system with 4~limit cycles has been constructed~\cite{Shi:1980:CEE}, we know that $H(2) \ge 4.$ Similarly, $H(3) \ge 13$, because cubic systems with at least 13 limit cycles have been found~\cite{Li:2009:CST,Yang:2010:ECT}.

Considering CRNs with two chemical species undergoing chemical reactions of at most $n$-th order, their reaction rate equations are given in the form~(\ref{xode})--(\ref{yode}), where $f(x,y)$ and $g(x,y)$ are real polynomials of degree at most $n$. In particular, if we denote by $C(n)$ the maximum number of stable limit cycles in such reaction rate equations, then we have $C(n) \le H(n).$ Considering CRNs with two chemical species undergoing chemical reactions of at most second order, it has been previously shown~\cite{Pota:1983:TBS, Schuman:2003:LCT} that their reaction rate equations cannot have any limit cycles (i.e. $C(2)=0$), while general ODEs with quadratic right-hand sides can have multiple limit cycles, with $H(2) \ge 4$. In particular, we observe that finding CRNs with two chemical species which have, under mass action kinetics, multiple stable limit cycles, is even more challenging than finding planar polynomial ODEs with multiple limit cycles. Considering cubic systems, we have $H(3) \ge 13$, but most of the chemical systems (with at most third-order reactions) in the literature often have at most one limit cycle~\cite{Field:1974:OCS,Schnakenberg:1979:SCR,Plesa:2016:CRS}. A chemical system with two stable limit cycles has been constructed~\cite{Plesa:2017:TMS}, giving $C(3) \ge 2$, but this is still much less than 13 limit cycles which can be found in some ODE systems with cubic right-hand sides in the literature~\cite{Li:2009:CST,Yang:2010:ECT}. To obtain multiple stable limit cycles in chemical systems, we have to consider higher-order chemical reactions or systems with more than two chemical species~\cite{Boros:2021:OPD,Boros:2022:LCM}.

Considering CRNs with N chemical species undergoing chemical reactions of at most $n$-th order, their reaction rate equations are given as the following system of ODEs
\begin{equation}
\frac{\mbox{d}{\mathbf x}}{\mbox{d}t}
=
{\mathbf f}({\mathbf x}),
\label{genODE}
\end{equation}
where ${\mathbf x} = (x_1,x_2,\dots,x_N) \in {\mathbb R}^N$ is the vector of concentrations of $N$ chemical species and its right-hand side ${\mathbf f}: {\mathbb R}^N \to {\mathbb R}^N$ is a vector of real polynomials of degree at most $n$. In this paper, we prove the following first main result:

\begin{theorem}
\label{thmcrn1}
Let $K$ be an arbitrary positive integer. Then there exists a CRN with $N(K)$ chemical species which are subject to $M(K)$ chemical reactions of at most seventh order such that \hfill\break
\rule{0pt}{1pt} \hskip 3mm {\rm (i)} reaction rate equations~$(\ref{genODE})$ have at least $K$ stable limit cycles,
\hfill\break 
\rule{0pt}{1pt} \hskip 3mm {\rm (ii)} we have $N(K) \le K+2$ and $M(K) \le 29 \, K$.
\end{theorem}

\noindent
Theorem~\ref{thmcrn1} provides a stronger result than finding $K$ limit cycles in a polynomial ODE system of the form~$(\ref{genODE})$, because not every polynomial ODE system corresponds to a CRN and, therefore, the set of reaction rate equations is a proper subset of ODEs with polynomial right-hand sides. To make the existence of $K$ limit cycles possible while restricting to polynomials of degree at most $n=7$, we allow for more than two chemical species, replacing the two ODE system~(\ref{xode})--(\ref{yode}) by a more general ODE system~(\ref{genODE}) with $N(K)$ equations. In particular, the next important question is how small the CRN can be so that it has $K$ limit cycles. Our answer is partially given in part (ii) of Theorem~\ref{thmcrn1} where we provide upper bounds on the number of chemical species involved and the number of chemical reactions (of at most seventh order). Another important parameter to consider is the maximum order of the chemical reactions involved, i.e. the degree $n$ of the polynomials on the right-hand side of ODE system~(\ref{genODE}). Since systems of at most second-order reactions (the case $n=2$) is of special interest in the theory of CRNs and applications~\cite{Wilhelm:2000:CSC}, we state our second main result as:

\begin{theorem}
\label{thmcrn2}
Let $K$ be an arbitrary positive integer. Then there exists a CRN with $N(K)$ chemical species which are subject to $M(K)$ chemical reactions of at most second order such that \hfill\break
\rule{0pt}{1pt} \hskip 3mm {\rm (i)} reaction rate equations~$(\ref{genODE})$ have at least $K$ stable limit cycles,
\hfill\break 
\rule{0pt}{1pt} \hskip 3mm {\rm (ii)} we have $N(K) \le 7K+14$ and $M(K) \le 42 \, K + 24$.
\end{theorem}

\noindent
By restricting to second-order (bimolecular) reactions, we obtain CRNs with more realistic second-order kinetics, but our construction increases the number of species and chemical reactions involved, as it can be seen by comparing parts (ii) of Theorems~\ref{thmcrn1} and~\ref{thmcrn2}. The precise definitions of CRNs, mass action kinetics, reaction rate equations and limit cycles in $N$-dimensional systems are given in Section~\ref{sec2}. 

In both Theorems~\ref{thmcrn1} and~\ref{thmcrn2}, we restrict our considerations to systems described by polynomial ODEs where the degree of polynomials is bounded by a constant independent of $K$, i.e. we consider polynomials of the degree at most $n=7$ (in Theorem~\ref{thmcrn1}) or at most $n=2$ (in Theorem~\ref{thmcrn2}), and we increase the number of chemical species, $N(K)$, as $K$ increases, to get $K$ stable limit cycles. Another approach is to restrict our considerations to chemical systems with only $N=2$ chemical species. In Section~\ref{sec8}, we construct two-species CRNs with $K$ stable limit cycles which include chemical reactions of at most $n(K)$-th order, where $n(K)=6K-2$. This establishes our third main result:

\begin{theorem}
\label{thmcrn3}
Let $C(n)$ be the maximum number of stable limit cycles of reaction rate equations~$(\ref{xode})$--$(\ref{yode})$ corresponding to CRNs with two chemical species undergoing chemical reactions of at most $n$-th order. Then we have
\begin{equation}
C(n) \ge \left\lfloor \frac{n+2}{6} \right\rfloor, 
\label{Cnbound}
\end{equation}
where the floor function $\lfloor . \rfloor$ denotes the integer part of a positive real number.
\end{theorem}

\noindent
To prove Theorems~\ref{thmcrn1}, \ref{thmcrn2} and~\ref{thmcrn3}, we first construct a planar system given by equations~(\ref{xode})--(\ref{yode}), where $f$ and $g$ are continuous non-polynomial functions chosen in such a way that the ODE system~(\ref{xode})--(\ref{yode}) has $K$ stable limit cycles in the positive quadrant $[0,\infty) \times [0,\infty)$. Such a planar non-polynomial ODE system is constructed in Section~\ref{sec3}. In Section~\ref{sec4}, we then increase the number of chemical species from~2 to $N(K$) to transform the non-polynomial ODE system to a polynomial one. In Section~\ref{sec5}, we modify this construction by using an $x$-factorable transformation to arrive at reaction rate equations corresponding to a CRN~\cite{Samardzija:1989:NCK}. Theorem~\ref{thmcrn1} is then proven in Section~\ref{sec6} by showing that the reaction rate equations have at least $K$ stable limit cycles. This is followed by our proof of Theorems~\ref{thmcrn2} and~\ref{thmcrn3} in Sections~\ref{sec7} and~\ref{sec8}, respectively.

\section{Notation and mathematical terminology}
\label{sec2}

\begin{definition}
\label{defcrn}
A {\it chemical reaction network} (CRN) is defined as a collection $(\mathcal{S},\mathcal{C},\mathcal{R})$ consisting of chemical species~$\mathcal{S}$, reaction complexes~$\mathcal{C}$ and chemical reactions~$\mathcal{R}$. We denote by $N$ the number of chemical species and by $M$ the number of chemical reactions, i.e. $\vert\mathcal{S}\vert=N$ and $\vert\mathcal{R}\vert=M$. Each chemical reaction is of the form
\begin{equation}
\label{eq:reaction}
\sum_{i=1}^N \nu_{i,j} X_i 
\;\, \longrightarrow \;\, 
\sum_{i=1}^N \nu'_{i,j} X_i,
\qquad \mbox{for} \; j=1,2,\dots,M,
\end{equation}
where $X_i,$ $i=1,2,\dots,N$, are chemical species, and linear combinations 
$\sum_{i=1}^N \nu_{i,j} X_i$ and $\sum_{i=1}^N \nu'_{i,j} X_i$ of species with non-negative integers $\nu_{i,j}$ and $\nu'_{i,j}$ are reaction complexes.
\end{definition}


\begin{definition}
\label{defrre}
Let $(\mathcal{S},\mathcal{C},\mathcal{R})$ be a CRN with $N$ chemical species and $M$ chemical reactions. Let $x_i(t)$ be the concentration of chemical species $X_i \in \mathcal{S}$, $i=1,2,\dots,N$. The time evolution of $x_i(t)$ is, under the assumption of the mass action kinetics, described by the {\it reaction rate equations}, which are written as a system of $N$ ODEs in the form
\begin{equation}
\frac{\mbox{d}x_i}{\mbox{d}t}(t)
=
\sum_{j=1}^M 
k_j
\,
(\nu'_{i,j} - \nu_{i,j})
\,
\prod_{\ell=1}^N x_\ell^{\nu_{\ell,j}},
\qquad
\mbox{for} \quad i=1,2,\dots,N,
\label{reactionrateequation}
\end{equation}
where $k_j$, $j=1,2,\dots,M$, is a positive constant called the reaction rate for the $j$-th reaction.
\end{definition}

\noindent
The reaction rate equations~(\ref{reactionrateequation}) are ODEs of the form~(\ref{genODE}), where the right-hand side ${\mathbf f}: {\mathbb R}^N \to {\mathbb R}^N$ is a vector of real polynomials. However, not every polynomial ODE system can be obtained as the reaction rate equations of a CRN as we formalize in the following Lemma.

\begin{lemma}
\label{lem1}
Consider a system of $N$ ODEs with polynomial right-hand sides describing the time evolution of $x_i(t)$, $i=1,2,\dots,N,$ in the form
\begin{equation}
\frac{\mbox{d}x_i}{\mbox{d}t}(t)
=
\sum_{j=1}^M 
\alpha_{i,j}
\,
\prod_{\ell=1}^N x_\ell^{\nu_{\ell,j}},
\qquad
\mbox{for} \quad i=1,2,\dots,N,
\label{reactionrateequationODEs}
\end{equation}
where $\alpha_{i,j}$ are real constants and $\nu_{i,j}$ are nonnegative integers, for $i=1,2,\dots,N$ and $j=1,2,\dots,M.$ Then the polynomial ODE system~$(\ref{reactionrateequationODEs})$ can be written as the reaction rate equations~$(\ref{reactionrateequation})$ of a CRN if and only if
\begin{equation}
\nu_{i,j} = 0
\quad
\mbox{implies}
\quad
\alpha_{i,j} \ge 0 
\quad
\mbox{for any $i=1,2,\dots,N$ 
and $j=1,2,\dots,M$}.
\label{nocrosstermcond}    
\end{equation}
\end{lemma}

\begin{proof}
The reaction rate equations~(\ref{reactionrateequation}) are of the form~(\ref{reactionrateequationODEs}). The non-negativity condition~(\ref{nocrosstermcond}) follows from $\nu_{i,j}=0$ and the non-negativity of both $k_j$ and $\nu'_{i,j}$ in equation~(\ref{reactionrateequation}). \hfill\break Conversely, if an ODE is of the form~(\ref{reactionrateequationODEs}) and $\alpha_{i,j}>0$, then we can choose $\nu'_{i,j}=\nu_{i,j}+1$ in equation~(\ref{reactionrateequation}) and put
the reaction rate as $k_j = \alpha_{i,j}$. On the other hand, if $\alpha_{i,j} < 0$, then the condition~(\ref{nocrosstermcond}) implies that
$\nu_{i,j} \ge 1$. Therefore, we can put
$\nu'_{i,j}=\nu_{i,j}-1$ and $k_j = -\alpha_{i,j}>0.$
\end{proof}

\noindent
In this paper, we prove the existence of limit cycles in chemical systems in Sections~\ref{sec6}, \ref{sec7} and~\ref{sec8} by proving the existence of limit cycles in systems of ODEs~(\ref{reactionrateequationODEs}) with polynomial right-hand sides satisfying the condition~(\ref{nocrosstermcond}). Then the approach used in the proof of Lemma~\ref{lem1} can be used to construct the corresponding CRN. However, the construction of a CRN corresponding to reaction rate equations is not unique. For example, consider a term of the form $-x_1^3$ on the right-hand side of equation~(\ref{reactionrateequationODEs}). Using the construction in the proof of Lemma~\ref{lem1}, it would correspond to the chemical reaction $3X \longrightarrow 2X$ with the rate constant equal to 1, but the same term can also correspond to the chemical reaction $3X \longrightarrow X$ with the rate constant equal to 1/2. We conclude this section by a formal definition of a stable limit cycle.

\begin{definition}
\label{deflc}
Consider a system of $N$ ODEs given by~(\ref{genODE}), where their right-hand side ${\mathbf f}: {\mathbb R}^N \to {\mathbb R}^N$ is continuous. A stable {\it limit cycle} is a trajectory $\mathbf{x}_{c}(t)$ for $t \in [0,\infty)$ such that \hfill\break
\rule{0pt}{1pt} \hskip 3mm {\rm (i)} 
$\mathbf{x}_{c}(t)$ is a solution of the ODE system~$(\ref{genODE})$,
\hfill\break 
\rule{0pt}{1pt} \hskip 3mm {\rm (ii)} 
there exists a positive constant $T>0$ such that
$\mathbf{x}_{c}(T)=\mathbf{x}_{c}(0)$ and \hfill\break 
\rule{0pt}{1pt} \hskip 10mm 
$\mathbf{x}_{c}(t) \ne \mathbf{x}_{c}(0)$ for $0 < t < T,$
\hfill\break 
\rule{0pt}{1pt} \hskip 3mm {\rm (iii)} 
there exists $\varepsilon>0$ such that
\hfill\break 
\rule{0pt}{1pt} \hskip 10mm 
$\mbox{dist}\{\mathbf{x}(0),\mathbf{x}_{c}\}
<\varepsilon$ 
\hskip 3mm
implies 
\hskip 3mm
$\mbox{dist}\{\mathbf{x}(t),\mathbf{x}_{c}\}
\to 0$ as $t \to \infty.$
\end{definition}

\noindent
In Definition~\ref{deflc}, constant $T$ is the period of the limit cycle and the property (iii) states that the limit cycle attracts all trajectories which start sufficiently close to it. The distance in the property (iii) of Definition~\ref{deflc} is the Euclidean distance defined by
$$
\mbox{dist}\{\mathbf{z},\mathbf{x}_{c}\}
=
\min_{t \in [0,T]}
\mbox{dist}\{\mathbf{z},\mathbf{x}_{c}(t)\}
=
\min_{t \in [0,T]}
\left(
\sum_{i=1}^N
\left(
z_i - x_{c,i}(t)
\right)^2
\right)^{1/2}
$$
for $\mathbf{z} = [z_1,z_2,\dots,z_N] \in {\mathbb{R}^N}$
and $\mathbf{x}_{c}(t) 
= [x_{c,1}(t),x_{c,2}(t),\dots,x_{c,N}(t)] \in {\mathbb{R}^N}$.

\section{Planar ODE systems with arbitrary number of limit cycles}
\label{sec3}

In this section, we construct a planar ODE system of the form~(\ref{xode})--(\ref{yode}) with $K$ limit cycles in the positive quadrant. It is constructed as a function of $2K$ parameters denoted by $a_1,$ $a_2,$ $\dots,$ $a_K$ and 
$b_1,$ $b_2,$ $\dots,$ $b_K,$ as
\begin{eqnarray}
\frac{\mbox{d}x}{\mbox{d}t} 
\!&=&\! 
\sum_{k=1}^K 
\frac{(x-a_k)
\big\{1-(x-a_k)^2-(y-b_k)^2\big\}
-(y-b_k)}{
1 + (x-a_k)^6 + (y-b_k)^6} 
\,=f(x,y)\,, \quad\; 
\label{xode2}\\
\frac{\mbox{d}y}{\mbox{d}t} 
\!&=&\! 
\sum_{k=1}^K 
\frac{(y-b_k)
\big\{1-(x-a_k)^2-(y-b_k)^2\big\}
+(x-a_k)}{
1 + (x-a_k)^6 + (y-b_k)^6} 
\,=g(x,y)\,. \quad\;
\label{yode2}
\end{eqnarray}
An illustrative dynamics of the ODE system~(\ref{xode2})--(\ref{yode2}) is shown in Figure~\ref{fig1}(a) for $K=4$,
\begin{figure}[t]%
(a) \hskip 5.86cm (b) \hfill\break
\centerline{
\includegraphics[width=0.495\textwidth]{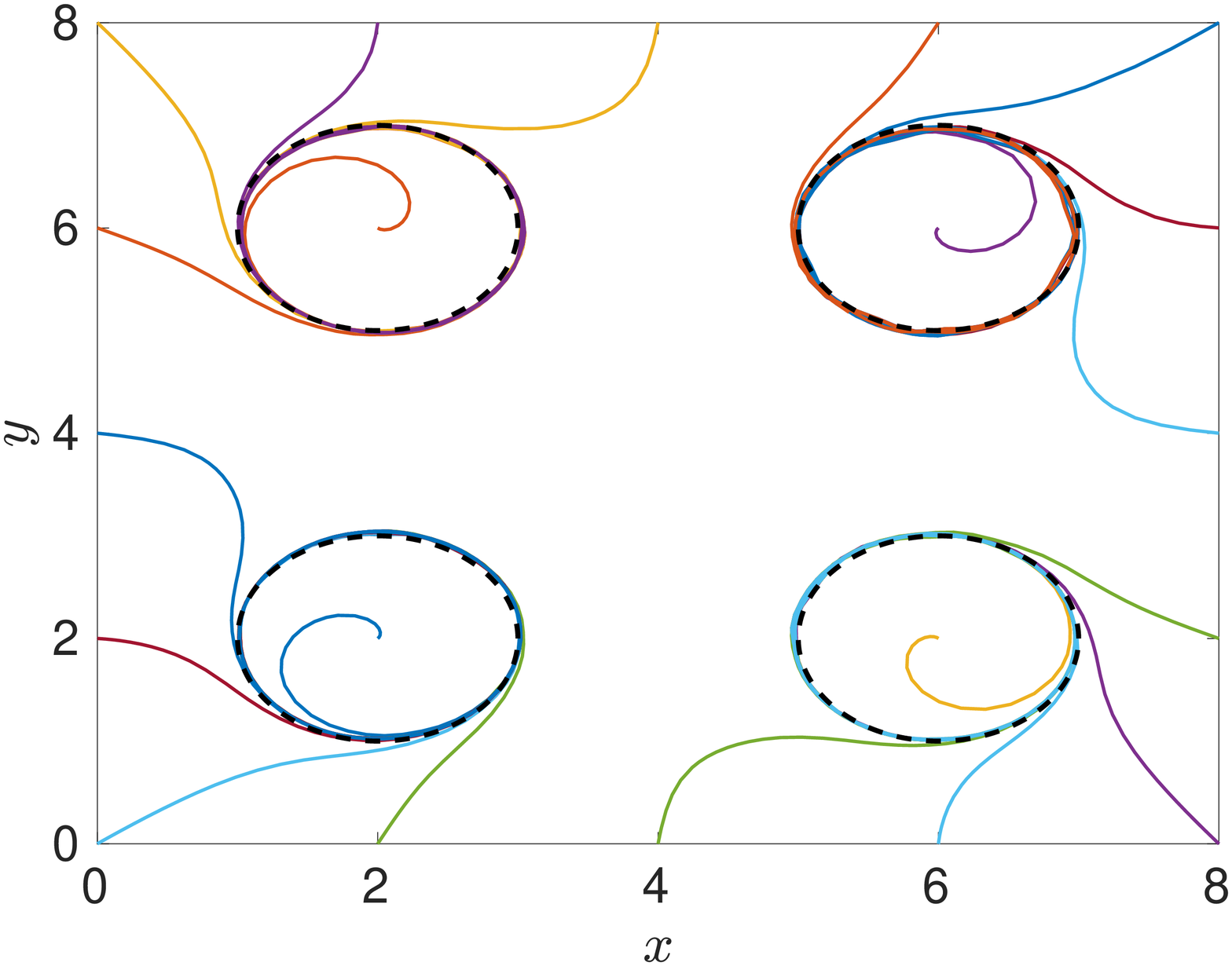}\;
\includegraphics[width=0.495\textwidth]{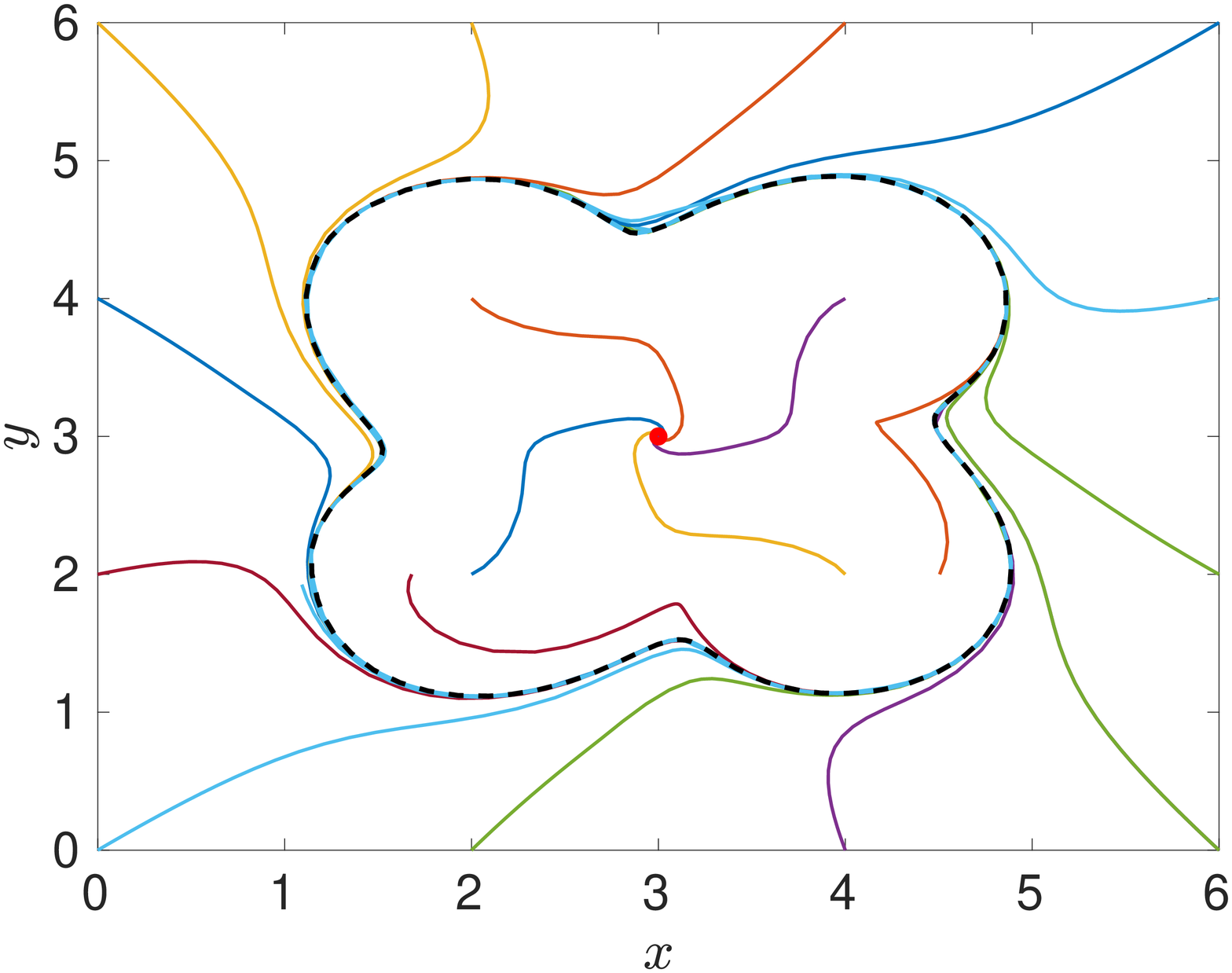}}
\caption{(a) {\it Twenty illustrative trajectories of the ODE system~$(\ref{xode2})$--$(\ref{yode2})$ for $K=4$ and the parameter choices $a_1 = b_1 = a_2 = b_3 = 2$ and $a_3 = b_2 = a_4 = b_4 = 6$. As $t \to \infty$, all presented trajectories approach one of the four limit cycles, which are plotted as the black dashed lines.} \hfill\break (b) {\it Twenty illustrative trajectories of the ODE system~$(\ref{xode2})$--$(\ref{yode2})$ for $K=4$ and the parameter choices $a_1 = b_1 = a_2 = b_3 = 2$ and $a_3 = b_2 = a_4 = b_4 = 4$. As $t \to \infty$, some trajectories converge to the stable limit cycle denoted by the black dashed line, while some trajectories, which started inside the limit cycle converge to the stable fixed point denoted as the red dot.}
}\label{fig1}
\end{figure}%
where the ODE system has four limit cycles, which is highlighted in Figure~\ref{fig1}(a) by plotting some representative trajectories. The existence of $K$ stable limit cycles for the ODE system~(\ref{xode2})--(\ref{yode2}) can also be proven analytically, as it is done in  Lemma~\ref{lem3}. In Figure~\ref{fig1}(a), we have presented an example with $K=4$ and parameter choices
$$
(a_1,b_1) = (2,2),
\quad
(a_2,b_2) = (2,6),
\quad
(a_3,b_3) = (6,2)
\quad
\mbox{and}
\quad 
(a_4,b_4) = (6,6). 
$$
In particular, the distance between points $(a_i,b_i)$, $i=1,2,3,4,$ is at least four. If we decrease this distance, then the ODE system~(\ref{xode2})--(\ref{yode2}) can have less limit cycles. This is highlighted in Figure~\ref{fig1}(b), where we present an example with $K=4$ and parameter choices
$$
(a_1,b_1) = (2,2),
\quad
(a_2,b_2) = (2,4),
\quad
(a_3,b_3) = (4,2)
\quad
\mbox{and}
\quad 
(a_4,b_4) = (4,4). 
$$
In Figure~\ref{fig1}(b), we observe that there is only one limit cycle, denoted as the black dashed line. This limit cycle is stable and a number of illustrative trajectories converge to this limit cycle as $t \to \infty$. However, there is also a stable equilibrium point at $(3,3)$, which attracts some of the trajectories. In particular, we can only expect that the ODE system~(\ref{xode2})--(\ref{yode2}) will have $K$ stable limit cycles provided that points $(a_i,b_i)$ are sufficiently separated. This is proven in our next lemma.

\begin{lemma}
\label{lem3}
Let us assume that
\begin{equation}
(a_i-a_j)^2+(b_i-b_j)^2 > 
15\left(K^{2/3}+2\right)
\qquad
\mbox{for all} \quad i \ne j,
\label{sepass}
\end{equation}
where $i,j=1,2,\dots,K.$ Then the ODE system~$(\ref{xode2})$--$(\ref{yode2})$ has at least $K$ stable limit cycles.
\end{lemma}

\begin{proof}
We define the sets
\begin{equation}
\Omega_i
=
\left\{
(x,y) \, : \,
1/2 < (x-a_i)^2 + (y-b_i)^2 < 2
\right\},
\quad\;
\mbox{for}
\;\; i=1,2,\dots,K.
\quad
\label{omegaidef}
\end{equation}
Then the condition~(\ref{sepass}) implies that 
$$
\Omega_i \cap \Omega_j = \emptyset,
\qquad
\mbox{for all} \quad i \ne j, 
\quad
\mbox{where}
\quad i,j=1,2,\dots,K,
$$
i.e. the sets $\Omega_i$ are pairwise disjoint sets. We will show that each of them contains at least one stable limit cycle. The boundary of $\Omega$ consists of two parts: outer and inner circles defined by
\begin{equation}
\partial \Omega_{i1} = \left\{ 
(x,y): (x - a_i)^2 + (y - b_i)^2 = 2
\right\}
\label{boundary1}
\end{equation}
and
\begin{equation}
\partial \Omega_{i2} = \left\{ 
(x,y): (x - a_i)^2 + (y - b_i)^2 = 1/2
\right\},
\label{boundary2}
\end{equation}
respectively, that is, $\partial\Omega_i =\partial \Omega_{i1} \cup \partial \Omega_{i2}$.
Define the following functions for $k=1,2,\dots,K$: 
\begin{eqnarray}
f_k(z_1,z_2)
&=&
\frac{ z_1 
\big\{
1-z_1^2-z_2^2
\big\}
-z_2}{
1 + z_1^6 + z_2^6
}, \label{f_k}\\
g_k(z_1,z_2)
&=&
\frac{z_2
\big\{1-z_1^2-z_2^2\big\}
+z_1}{
1 + z_1^6 + z_2^6}.\label{g_k}
\end{eqnarray}
Then, the ODE system~(\ref{xode2})--(\ref{yode2}) can be rewritten as
\begin{eqnarray}
\frac{\mbox{d}x}{\mbox{d}t} 
\!&=&\! 
f(x,y)
\,, \quad\; 
\label{xode2_2}\\
\frac{\mbox{d}y}{\mbox{d}t} 
\!&=&\! 
g(x,y)
\,, \quad\;
\label{yode2_2}
\end{eqnarray}
where
\begin{equation}
f(x,y)
=
\sum_{k=1}^K 
f_k(x-a_k,y-b_k)
\quad
\mbox{and}
\quad
g(x,y)
=
\sum_{k=1}^K g_k(x-a_k,y-b_k).
\label{fgdef}
\end{equation}
First, we will show that $\Omega_i$ for $i=1,2,\dots, K$ does not contain any equilibrium points. Let us consider any point $(x^*,y^*)\in\Omega_i$. Substituting  
\begin{equation}
x^*=a_i+r\cos{\theta},
\quad
y^*=b_i+r\sin{\theta},
\label{polar_equilibrium}
\end{equation}
in the terms for $k=i$ in (\ref{fgdef}), we obtain
\begin{eqnarray}
f(x^*,y^*)
\!&=&\! 
\frac{r\cos{\theta} \, \{1-r^2\}-r\sin{\theta}}{1+r^6\cos^6{\theta}+r^6\sin^6{\theta}}
\,
+
\sum_{k=1,k\neq i}^K 
f_k(x^*-a_k,y^*-b_k)
\,, \quad\; 
\label{xode2_3}\\
g(x^*,y^*) 
\!&=&\! 
\frac{r\sin{\theta}\, \{1-r^2\}+r\cos{\theta}}{1+r^6\cos^6{\theta}+r^6\sin^6{\theta}}
\,
+
\sum_{k=1,k\neq i}^K g_k(x^*-a_k,y^*-b_k)
\,. \quad\;
\label{yode2_3}
\end{eqnarray}
The first terms in (\ref{xode2_3})--(\ref{yode2_3}) can be rewritten as 
\begin{equation}
\frac{4 r\sqrt{(r^2-1)^2+1}\,\,
\sin(\theta+\tilde{\theta})}
{4 +r^6\left(4- 3 \sin^2{2\theta}\right)} \, ,
\label{ode2_4}
\end{equation}
where $\tilde{\theta}=\alpha$ with $\tan{\alpha}=r^2-1$  and $\pi/2 < \alpha < 3\pi/2$ in the case of~(\ref{xode2_3}) and $\tilde{\theta}=\alpha-\pi/2$ in the case of~(\ref{yode2_3}). Since we have
$$
\max
(
\vert\sin(\theta+\alpha)\vert,
\vert\sin(\theta+\alpha-\pi/2)\vert
)
>
1/\sqrt{2}
$$
for any $\theta$ and $\alpha$,
at least one of the two terms expressed in the form~(\ref{ode2_4}) is greater than 
$$
\frac{1}{\sqrt{2}}\frac{r\sqrt{(r^2-1)^2+1}}
{1+r^6} \, ,
$$
which has a minimum $\sqrt{2}/9$ when $1/2<r^2<2$. Therefore, at least one of
the absolute values of the $i$-th components, $f_i(x^*-a_i,y^*-b_i)$ and $g_i(x^*-a_i,y^*-b_i)$, in (\ref{xode2_3})--(\ref{yode2_3}) at any point $(x^*,y^*)\in\Omega_i$ is greater than equal to $\sqrt{2}/9$. Without loss of generality, we assume
$$
\vert f_i(x^*-a_i,y^*-b_i) \vert
\ge
\vert g_i(x^*-a_i,y^*-b_i) \vert.
$$
Then we have  $\vert f_i(x^*-a_i,y^*-b_i) \vert \ge \sqrt{2}/9$.
We want to show that the second term in (\ref{xode2_3}) (i.e. the sum) has a smaller magnitude than the first term $f_i(x^*-a_i,y^*-b_i)$ so that we could conclude that $f(x^*,y^*)\neq 0$.
The $k$-th component in the second term in (\ref{xode2_3}) is bounded by
\begin{eqnarray}
\left\vert f_k(z_1,z_2)\right\vert 
&\le &
 \frac{\vert z_1\vert \,
 \vert 1-z_1^2-z_2^2\vert +\vert z_2\vert}{\vert 1+z_1^6+z_2^6\vert }
\label{fk_est1}
\end{eqnarray}
where $(z_1,z_2)=(x^*-a_k,y^*-b_k)$.
Denoting $c^2 = z_1^2+z_2^2$, we have
\begin{equation}
1+\frac{c^6}{4}\le 1+z_1^6+z_2^6\le 1+c^6. \label{c6_bound}
\end{equation}
Using
$\vert z_i\vert \le c$ and (\ref{c6_bound}), we estimate (\ref{fk_est1}) as
\begin{equation}
\left\vert f_k(z_1,z_2)\right\vert 
\le
\frac{c \left(\lvert 1-c^2\rvert +1\right)}{1+c^6/4 } \, . 
\label{fk_est2}
\end{equation}
Since $(x^*,y^*)\in \Omega_i$ and $(a_k,b_k)\in\Omega_k$ where $k\neq i$, our assumption~(\ref{sepass}) implies that $c^2\ge 2$. Thus, (\ref{fk_est2}) becomes
\begin{equation}
\left\vert f_k(z_1,z_2)\right\vert 
\le
\frac{c^3}{1+c^6/4} 
\le \frac{4}{c^3} \, . 
\label{fk_est3}
\end{equation}
Therefore, the magnitude of the second term in (\ref{xode2_3}) is bounded by $4(K-1)/c^3$.
Since $\vert f_i(x^*-a_i,y^*-b_i) \vert \ge \sqrt{2}/9$,
a sufficient condition
for $f(x^*,y^*)\neq 0$ 
is 
\begin{equation}
\frac{4(K-1)}{c^3} < \frac{\sqrt{2}}{9} \, .
\label{lemma3_condition1}
\end{equation}
Using the assumption~(\ref{sepass}), 
the distance  $c=\sqrt{(x^*-a_k)^2+(y^*-b_k)^2}$ is bounded by
\begin{equation}
c > \sqrt{(a_i-a_k)^2+(b_i-b_k)^2}-\sqrt{2}
> \sqrt{15\left(K^{2/3}+2\right)} -\sqrt{2} \, ,
\label{c_bound}
\end{equation}
which implies the sufficient condition~(\ref{lemma3_condition1}). 
Therefore, $(x^*,y^*)$ is not an equilibrium point.

Next, consider an arbitrary point $(x_b,y_b) \in \partial \Omega_{i1}$. Let us calculate the scalar product of vectors
\begin{equation}
(x_b-a_i,y_b-b_i)
\qquad
\mbox{and}
\qquad
\big(f(x_b,y_b), g(x_b,y_b)\big) \, .
\label{scalar_vectors}
\end{equation}
Using~(\ref{fgdef}), 
we obtain this scalar product as
\begin{eqnarray}
&& (x_b - a_i) 
f_i(x_b-a_i,y_b-b_i) 
+
(y_b - b_i)
g_i(x_b-a_i,y_b-b_i) \label{scalar_product}\\
&&\hspace{-1cm} +
(x_b - a_i) \!\!\!\!
\sum_{k=1,k \ne i}^K
f_k(x_b-a_k,y_b-b_k) 
+
(y_b - b_i) \!\!\!\!
\sum_{k=1, k \ne i}^K
g_k(x_b-a_k,y_b-b_k) \, . 
\nonumber
\end{eqnarray}
The first two terms in (\ref{scalar_product}) become 
$$
\frac{-2}{1+(x_b-a_i)^6+(y_b-a_i)^6} \, ,
$$
which has a magnitude greater than $2/9$ using (\ref{c6_bound}) with $c^2=(x_b-a_i)^2+(y_b-b_i)^2=2$.
Using (\ref{fk_est3}), $\vert x_b-a_i\vert \le \sqrt{2}$ and $\vert y_b-b_i\vert \le \sqrt{2}$, we can estimate the third and fourth terms in~(\ref{scalar_product}), namely, we have
\begin{equation}
\left\vert (x_b-a_i) \, f_k(z_1,z_2)\right\vert 
\le
\frac{4\sqrt{2}}{d^3} 
\quad
\mbox{and}
\quad
\left\vert (y_b-b_i) \, g_k(z_1,z_2)\right\vert 
\le
\frac{4\sqrt{2}}{d^3} \, ,
\label{fk_est4}
\end{equation}
where $d^2=(x_b-a_k)^2+(y_b-b_k)^2$.
Then the sum of the third and fourth terms in (\ref{scalar_product}) is bounded by
$8\sqrt{2}(K-1)/d^3$. Therefore, the sufficient condition that the scalar product in (\ref{scalar_product}) is negative is
\begin{equation}
\frac{8\sqrt{2}(K-1)}{d^3} < \frac{2}{9}.
\label{lemma3_condition2}
\end{equation}
Using the assumption~(\ref{sepass}), 
the distance  $d=\sqrt{(x_b-a_k)^2+(y_b-b_k)^2}$ is bounded by
\begin{equation}
d > \sqrt{(a_i-a_k)^2+(b_i-b_k)^2}-\sqrt{2}
> \sqrt{15\left(K^{2/3}+2\right)} -\sqrt{2} \, ,
\label{destimates}
\end{equation}
which implies the sufficient condition~(\ref{lemma3_condition2}). Therefore, the vector
$$
\big(f(x_b,y_b), g(x_b,y_b)\big)
$$
always points inside the domain $\Omega_i$ for each boundary point $(x_b,y_b) \in \partial \Omega_{i1}$. 

Similarly, for an arbitrary point $(x_b,y_b)\in \partial \Omega_{i2}$, we can show that the scalar product of vectors in (\ref{scalar_vectors}) is always positive due to that the sum of the first two terms in (\ref{scalar_product}) is equal to
$$
\frac{1/4}{1+(x_b-a_i)^6+(y_b-b_i)^6} \, ,
$$
which is greater than $2/9$ by using~(\ref{c6_bound})
with $c^2=1/2$, and the sum of the third and fourth terms in (\ref{scalar_product}) is bounded by
$8(K-1)/(d^3 \sqrt{2})$.  Therefore, the sufficient condition that the scalar product in (\ref{scalar_product}) is positive is
$$
\frac{8(K-1)}{d^3 \sqrt{2}} < \frac{2}{9} \, ,
$$
which is a weaker condition then the condition~(\ref{lemma3_condition2}), i.e. it is again
satisfied because of our assumption~(\ref{sepass}). This
implies that the scalar product in (\ref{scalar_product}) is positive. Thus, the directional vector always points inside the domain $\Omega_i$ on all parts of the boundary $\partial \Omega_i$.

Therefore, applying Poincar\'e-Bendixson theorem, we conclude that each $\Omega_i$ contains at least one stable limit cycle. Since $\Omega_i$, $i=1,2,\dots,K,$ are pairwise disjoint, this implies that the ODE system~(\ref{xode2})--(\ref{yode2}) has at least $K$ stable limit cycles.
\end{proof}

\section{ODE systems with polynomial right-hand sides and arbitrary number of limit cycles}
\label{sec4}

\noindent
Considering an auxiliary variable
\begin{equation}
u_i=\frac{1}{1 + (x-a_i)^6 + (y-b_i)^6},
\qquad
\mbox{for}
\quad
i=1,2,\dots,K,
\label{nonpolfactor}
\end{equation}
we can formally convert the non-polynomial ODE 
system~(\ref{xode2})--(\ref{yode2}) to a system of $(K+2)$ ODEs with polynomial right-hand sides~\cite{Kerner:1981:UFN}. We obtain
\begin{eqnarray}
\frac{\mbox{d}x}{\mbox{d}t} 
\!&=&\! 
\sum_{k=1}^K u_k \left[(x-a_k)
\big\{1-(x-a_k)^2-(y-b_k)^2\big\}
-(y-b_k)\right], \quad\; 
\label{xode3}\\
\frac{\mbox{d}y}{\mbox{d}t} 
\!&=&\! 
\sum_{k=1}^K u_k \left[(y-b_k)
\big\{1-(x-a_k)^2-(y-b_k)^2\big\}
+(x-a_k)\right], \quad\;
\label{yode3} \\
\frac{\mbox{d}u_i}{\mbox{d}t}
\!&=&\! 
- 6 u_i^2 (x - a_i)^5
\sum_{k=1}^K u_k \left[(x-a_k)
\big\{1-(x-a_k)^2-(y-b_k)^2\big\}
-(y-b_k)\right] \quad \; 
\nonumber\\
\!&-&\!
6 u_i^2 (y - b_i)^5
\sum_{k=1}^K u_k \left[(y-b_k)
\big\{1-(x-a_k)^2-(y-b_k)^2\big\}
+(x-a_k)\right], \quad\;
\label{uiode3}
\end{eqnarray}
for $i=1,2,\dots,K$. The dynamics of the original ODE system~~(\ref{xode2})--(\ref{yode2}) with initial condition $(x(0),y(0))=(x_0,y_0)$ is the same as the dynamics of the extended ODE system~(\ref{xode3})--(\ref{uiode3}), when we initialize the additional variables by
\begin{equation}
u_i(0) 
=
\frac{1}{1 + (x_0-a_i)^6 + (y_0-b_i)^6},
\qquad
\mbox{for}
\quad
i=1,2,\dots,K.
\label{initcond1}
\end{equation}
However, when we use a general initial condition,
$$
(x(0),y(0),u_1(0),u_2(0),\dots,u_K(0)) \in {\mathbb R}^{K+2},
$$
the trajectory of the extended ODE system~(\ref{xode3})--(\ref{uiode3}) may become unbounded and it may not converge to a limit cycle. To illustrate this behaviour, let us consider the initial condition 
\begin{equation}
u_i(0) 
=
\frac{c}{1 + (x_0-a_i)^6 + (y_0-b_i)^6},
\qquad
\mbox{for}
\quad
i=1,2,\dots,K,
\label{initcond2}
\end{equation}
where $c>0$ is a constant. If $c=1$, then the initial condition~(\ref{initcond2}) reduces to~(\ref{initcond1}). In particular, Figure~\ref{fig1}(a) shows an illustrative behaviour of both the extended ODE system~(\ref{xode3})--(\ref{uiode3}) for $c=1$ and the planar ODE system~(\ref{xode2})--(\ref{yode2}), assuming that we use a sufficiently accurate numerical method to solve ODEs~(\ref{xode3})--(\ref{uiode3}) and plot the projection of the calculated trajectory to the $(x,y)$-plane. Changing $c=1$ to $c=0.5$, we plot the dynamics of the extended ODE system in Figure~\ref{fig2}(a), where the black dots denote the end points of the calculated trajectories at the final time ($t=100$). We observe that only the trajectories which started `inside a limit cycle' (i.e. their projections to the $(x,y)$-plane are initially inside a black dashed circle) seem to converge to it, while the other trajectories do not seem to approach the `limit cycles'. This is indeed the case even if we continue these trajectories for times $t>100.$ In fact, depending on the accuracy of the numerical method used, all trajectories eventually stop somewhere in the phase plane, because $u_i(t) \to 0$ as $t \to \infty$. 

\begin{figure}[t]%
(a) \hskip 5.86cm (b) \hfill\break
\centerline{
\includegraphics[width=0.495\textwidth]{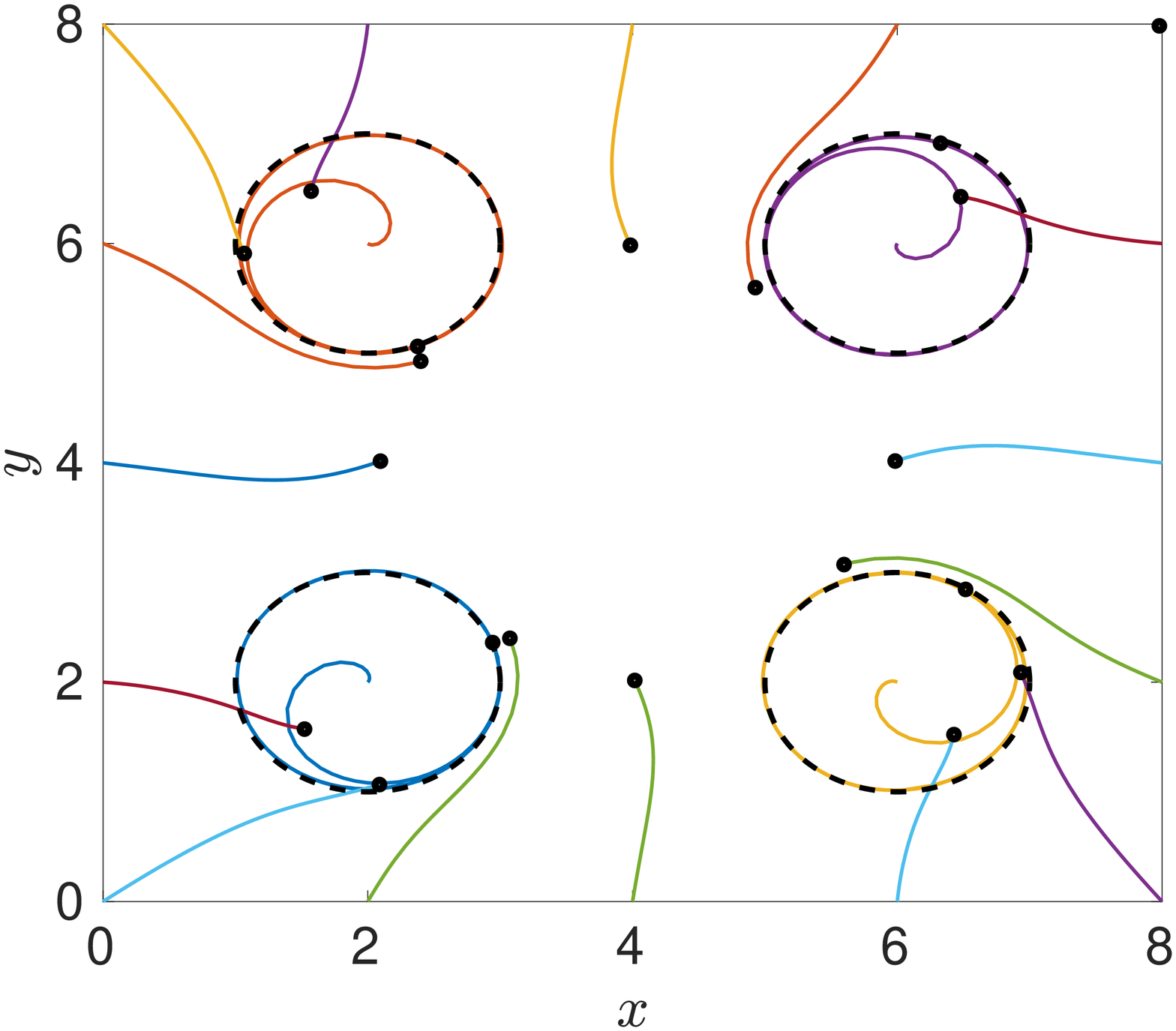}\;
\includegraphics[width=0.495\textwidth]{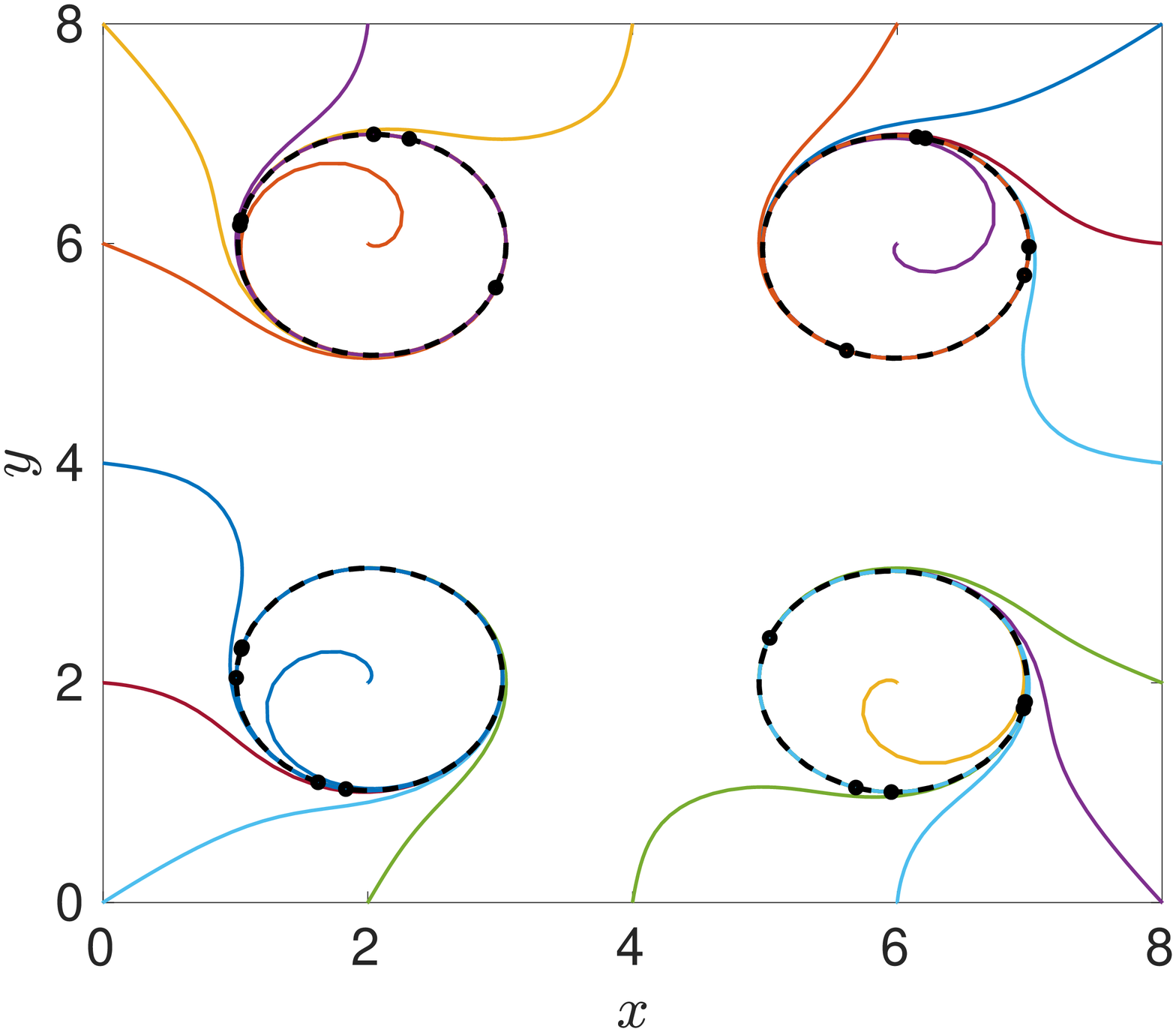}}
\caption{(a) {\it Twenty illustrative trajectories of the ODE system~$(\ref{xode3})$--$(\ref{uiode3})$ for $K=4$, the parameter choices $a_1 = b_1 = a_2 = b_3 = 2$, $a_3 = b_2 = a_4 = b_4 = 6$ and the initial condition~$(\ref{initcond2})$ with $c=1/2$. The black dots denote the final position of each calculated trajectory at time $t=100$. The black dashed lines are limit cycles shown in Figure~$\ref{fig1}(a)$.} \hfill\break (b) {\it Twenty illustrative trajectories of the ODE system~$(\ref{xode4})$--$(\ref{viode4})$ for $K=4$, the parameter choices $\varepsilon=1$, $a_1 = b_1 = a_2 = b_3 = 2$, $a_3 = b_2 = a_4 = b_4 = 6$ and the initial condition~$(\ref{initcond2})$ with $c=1/2$. As $t \to \infty$, all trajectories approach one of the four limit cycles, which are plotted as the black dashed lines. The black dots denote the final position of each calculated trajectory at time $t=100$.}
}\label{fig2}
\end{figure}

On the other hand, considering the extended ODE system~(\ref{xode3})--(\ref{uiode3}) with the initial condition (\ref{initcond2}) for $c>1$, some additional variables $u_i(t)$ tend to infinity as $t \to \infty$, and we again do not observe sustained oscillations in our numerical experiments (results not shown). In particular, the formal conversion of the non-polynomial ODE system~(\ref{xode2})--(\ref{yode2}) into the polynomial system~(\ref{xode3})--(\ref{uiode3}) does not preserve the dynamics well.
Therefore, we augment our polynomial ODE system~(\ref{xode2})--(\ref{yode2}) in a different way. We introduce $K$ new variables $v_i$,
$i=1,2,\dots,K$, and formulate the extended ODE system as the following $(K+2)$ equations:
\begin{eqnarray}
\frac{\mbox{d}x}{\mbox{d}t} 
\!&=&\! 
\sum_{k=1}^K v_k \left[(x-a_k)
\big\{1-(x-a_k)^2-(y-b_k)^2\big\}
-(y-b_k)\right], \quad\; 
\label{xode4}\\
\frac{\mbox{d}y}{\mbox{d}t} 
\!&=&\! 
\sum_{k=1}^K v_k \left[(y-b_k)
\big\{1-(x-a_k)^2-(y-b_k)^2\big\}
+(x-a_k)\right], \quad\;
\label{yode4} \\
\varepsilon \, \frac{\mbox{d}v_i}{\mbox{d}t}
\!&=&\! 
1 - v_i \left[ 1 + (x - a_i)^6 + (y - b_i)^6 \right],
\qquad
\mbox{for}
\quad
i=1,2,\dots,K,
\label{viode4}
\end{eqnarray}
where $\varepsilon > 0$ is a constant.
The first two ODEs~(\ref{xode4})--(\ref{yode4}) are the same as ODEs~(\ref{xode3})--(\ref{yode3}) with $v_k$ taking place of~$u_k$. The difference is in the dynamics of the additional variables, i.e. in equation~(\ref{viode4}) which removes the non-polynomial factor~(\ref{nonpolfactor}) in a different way. Rather than defining new variable $u_i$ in the form~(\ref{nonpolfactor}) and deriving ODEs which have equivalent dynamics to the ODE system~(\ref{xode2})--(\ref{yode2}) for a very special initial condition~(\ref{initcond1}), we
have written the ODE~(\ref{viode4}) in such a way that it formally recovers the non-polynomial factor~(\ref{nonpolfactor}) in the limit $\varepsilon \to 0$, which will be used in our proof of Lemma~\ref{lem4}, where we consider small values of $\varepsilon$. 
However, even for larger values of $\varepsilon$, the ODE system~(\ref{xode4})--(\ref{viode4}) has multiple limit cycles for general initial conditions, as it is illustrated for $\varepsilon=1$ and $K=4$ in Figure~\ref{fig2}(b), where all plotted trajectories finish on a limit cycle (see the final calculated positions, at time $t=100$, plotted as black dots). 

Next, we prove that the extended system~(\ref{xode4})--(\ref{viode4}) has $K$ limit cycles 
in the sense of Definition~\ref{deflc} 
for general values of~$K$. Since~(\ref{xode4})--(\ref{viode4}) is a system of $(K+2)$ ODEs, we cannot directly apply the Poincar\'e-Bendixson theorem as we did for the planar system in the proof of Lemma~\ref{lem3}. While one possible approach to proving the existence of limit cycles is to work with generalizations of the Poincar\'e-Bendixson theorem to higher dimensional ODEs~\cite{Hirsch:1982:SDE,Li:1996:PAS,Sanchez:2010:EPO}, we will base our proof of Lemma~\ref{lem4} on the application of Tikhonov's theorem~\cite{Tikhonov:1952:SDE,Klonowski:1983:SPC} and the result of Lemma~\ref{lem3}. In particular, we show that the extended system~(\ref{xode4})--(\ref{viode4}) is a polynomial system which has $K$ limit cycles for sufficiently small values of $\varepsilon.$

\begin{lemma}
\label{lem4}
Let us assume that parameters $a_i>0$ and $b_i>0$, $i=1,2,\dots,K,$ satisfy the inequality~$(\ref{sepass})$.
Then there exists $\varepsilon_0>0$ such that the ODE system~$(\ref{xode4})$--$(\ref{viode4})$ has at least $K$ stable limit cycles for all $\varepsilon \in (0,\varepsilon_0).$
\end{lemma}

\begin{proof}
Let us consider $\varepsilon=0$. Then the right-hand side of the ODE~(\ref{viode4}) is equal to zero. This equation can be solved for $v_i$, $i=1,2,\dots,K,$ to obtain 
$v_i = q_i(x,y),$ where we define
\begin{equation}
q_i(x,y) 
= 
\frac{1}{1 + (x - a_i)^6 + (y - b_i)^6}.
\label{defqi}
\end{equation}
 Substituting $v_i=q_i(x,y)$ into (\ref{xode4})--(\ref{yode4}), we obtain that the reduced problem in the sense of Tikhonov's theorem~\cite{Tikhonov:1952:SDE,Klonowski:1983:SPC} is equal to
\begin{eqnarray}
\frac{\mbox{d}\overline{x}}{\mbox{d}t} 
&=&
f(\overline{x},\overline{y})\,, \quad\; 
\label{xode2red}\\
\frac{\mbox{d}\overline{y}}{\mbox{d}t} 
&=& 
g(\overline{x},\overline{y})\,, \quad\;
\label{yode2red}
\end{eqnarray}
where functions $f(\cdot,\cdot)$ and $g(\cdot,\cdot)$ are defined in (\ref{xode2}) and (\ref{yode2}). This means that the reduced system~(\ref{xode2red})--(\ref{yode2red}) corresponding to the fast--slow extended ODE system~(\ref{xode4})--(\ref{viode4}) is the same as our original non-polynomial ODE system (\ref{xode2})--(\ref{yode2}). Therefore, using Lemma~\ref{lem3}, we know that the reduced system~(\ref{xode2red})--(\ref{yode2red}) has (at least) $K$ stable limit cycles in the sense of Definition~\ref{deflc}, i.e. there exist $K$ solutions
\begin{equation} 
(\overline{x}_{c,i}(t),
\overline{y}_{c,i}(t))
\qquad
\mbox{for} 
\quad
t \in [0,\infty),
\quad
i = 1,2,\dots,K,
\label{limitcycles}
\end{equation}
of the reduced system~(\ref{xode2red})--(\ref{yode2red})
which are periodic with period $T_i>0$ for $i=1,2,\dots,K$. Moreover, there exist $\varepsilon_i>0$, $i=1,2,\dots,K$, such that any solution 
$(\overline{x}(t),\overline{y}(t))$ of the reduced system~(\ref{xode2red})--(\ref{yode2red}) approaches the limit cycle $(\overline{x}_{c,i}(t),
\overline{y}_{c,i}(t))$ as $t \to \infty$ provided that the initial condition $(\overline{x}(0),\overline{y}(0))$ satisfies
\begin{equation}
\min_{t \in [0,T_i]}
\big(\overline{x}(0)-\overline{x}_{c,i}(t)\big)^2
+
\big(\overline{y}(0)-\overline{y}_{c,i}(t)\big)^2
<\varepsilon_i.
\label{epsidef}
   \end{equation}
Next, we define pairwise disjoint sets $\Omega_i \subset \mathbb{R}^{K+2}$ for $i=1,2,\dots,K$  by
\begin{eqnarray}
\Omega_i 
& = &
\bigg\{
(x,y,v_1,v_2,\dots,v_K) \in \mathbb{R}^{K+2}
\quad \mbox{such that}
\label{defomegailem4}    
\\
&&
\min_{t \in [0,T_i]}
\big(x-\overline{x}_{c,i}(t)\big)^2
+
\big(y-\overline{y}_{c,i}(t)\big)^2
+
\sum_{j=1}^K
\big(v_j
-
q_j(\overline{x}_{c,i}(t),\overline{y}_{c,i}(t))\big)^2
<\varepsilon_i
\bigg\} \, ,
\nonumber
\end{eqnarray}
where functions $q_j(\cdot,\cdot)$ are defined by~(\ref{defqi}). Let us define
$$
\varepsilon_0
=
\min_{i \in \{1, 2, \dots, K \} }
\varepsilon_i.
$$
Let $\varepsilon \in (0,\varepsilon_0)$ be chosen arbitrarily. To show that the extended fast-slow polynomial ODE system~(\ref{xode4})--(\ref{viode4}) has (at least) $K$ stable limit cycles, it is sufficient to show that each set $\Omega_i$ contains one stable limit cycle. We will do this by applying Tikhonov's theorem~\cite{Tikhonov:1952:SDE,Klonowski:1983:SPC}. Considering the ODEs~(\ref{viode4}) for $i=1,2,\dots,K,$ where $x>0$ and $y>0$ are taken as parameters, we obtain the adjoined system as a $K$-dimensional system of ODEs with an isolated stable equilibrium point $[q_1(x,y),q_2(x,y),\dots,q_K(x,y)]$, where $q_i(\cdot,\cdot)$ is defined in (\ref{defqi}). This equilibrium point attracts the solutions of adjoined system for any initial condition. Therefore, the ODE system (\ref{xode4})--(\ref{viode4}) has a limit cycle in~$\Omega_i$. Moreover, this limit cycle attracts any solution  $\big(x(t),y(t),v_1(t),v_2(t),\dots,v_K(t)\big)$ of the system~(\ref{xode4})--(\ref{viode4}) with initial condition satisfying
$
\big(x(0),y(0),v_1(0),v_2(0),\dots,v_K(0)\big) \in \Omega_i.
$
\end{proof}

\section{Chemical systems with arbitrary many limit cycles}
\label{sec5}

To construct a CRN with $K$ limit cycles, we first construct a system of ODEs with polynomial right-hand sides which satisfy the condition~(\ref{nocrosstermcond}) in Lemma~\ref{lem1}, i.e. it will be a system of reaction rate equations, which correspond to a CRN. Once we have such reaction rate equations, there are infinitely many CRNs described by them, so we conclude this section by specifying some illustrative CRNs corresponding to the derived  reaction rate equations. 

Our starting point is the polynomial ODE system~(\ref{xode4})--(\ref{viode4}), which has $K$ limit cycles provided that the conditions of Lemma~\ref{lem4} are satisfied. The reaction rate equations are constructed by applying the so called $x$-factorable transformation~\cite{Plesa:2016:CRS} to the right-hand sides of equations (\ref{xode4}) and (\ref{yode4}). We do not modify the right-hand sides of ODEs~(\ref{viode4}), because they already satisfy the conditions of Definition~\ref{defrre}. We obtain the ODE system:
\begin{eqnarray}
\frac{\mbox{d}x}{\mbox{d}t} 
\!&=&\! 
\sum_{k=1}^K x \, v_k \left[(x-a_k)
\big\{1-(x-a_k)^2-(y-b_k)^2\big\}
-(y-b_k)\right], \quad\; 
\label{xode5}\\
\frac{\mbox{d}y}{\mbox{d}t} 
\!&=&\! 
\sum_{k=1}^K y \, v_k \left[(y-b_k)
\big\{1-(x-a_k)^2-(y-b_k)^2\big\}
+(x-a_k)\right], \quad\;
\label{yode5} \\
\varepsilon \, \frac{\mbox{d}v_i}{\mbox{d}t}
\!&=&\! 
1 - v_i \left[ 1 + (x - a_i)^6 + (y - b_i)^6 \right],
\qquad
\mbox{for}
\quad
i=1,2,\dots,K.
\label{viode5}
\end{eqnarray}
The illustrative dynamics of the ODE system~(\ref{xode5})--(\ref{viode5}) is presented in Figure~\ref{fig3}(a), where we use the same parameters as we use in Figure~\ref{fig2}(b) for the ODE system~(\ref{xode4})--(\ref{viode4}). We observe that the presented trajectories converge to one of the four limit cycles as in Figure~\ref{fig2}(b). The shape of the limit cycles is slightly modified by using the $x$-factorable transformation, but the limit cycles are still there as we formally prove in Section~\ref{sec6}. 

The $x$-factorable transformations modify the dynamics on the $x$-axis and $y$-axis. In Figure~\ref{fig3}(a), we present illustrative trajectories which all start with the positive values of $x(0)$ and $y(0)$, while in Figure~\ref{fig2}(b), some of the illustrative trajectories have zero initial values of $x(0)$ and $y(0)$. To get a comparable plot, we use the same initial conditions in both Figure~\ref{fig2}(b) and Figure~\ref{fig3}(a), with the only exception that all initial conditions with $x(0)=0$ (resp. $y(0)=0$) in Figure~\ref{fig2}(b) are replaced by $x(0)=1/2$
(resp. $y(0)=1/2$) in Figure~\ref{fig3}(a). We note that if we start a trajectory of
the ODE system~(\ref{xode5})--(\ref{viode5}) on the $x$-axis or the $y$-axis, then it stays on the axis.

\begin{figure}[t]%
(a) \hskip 5.86cm (b) \hfill\break
\centerline{
\includegraphics[width=0.495\textwidth]{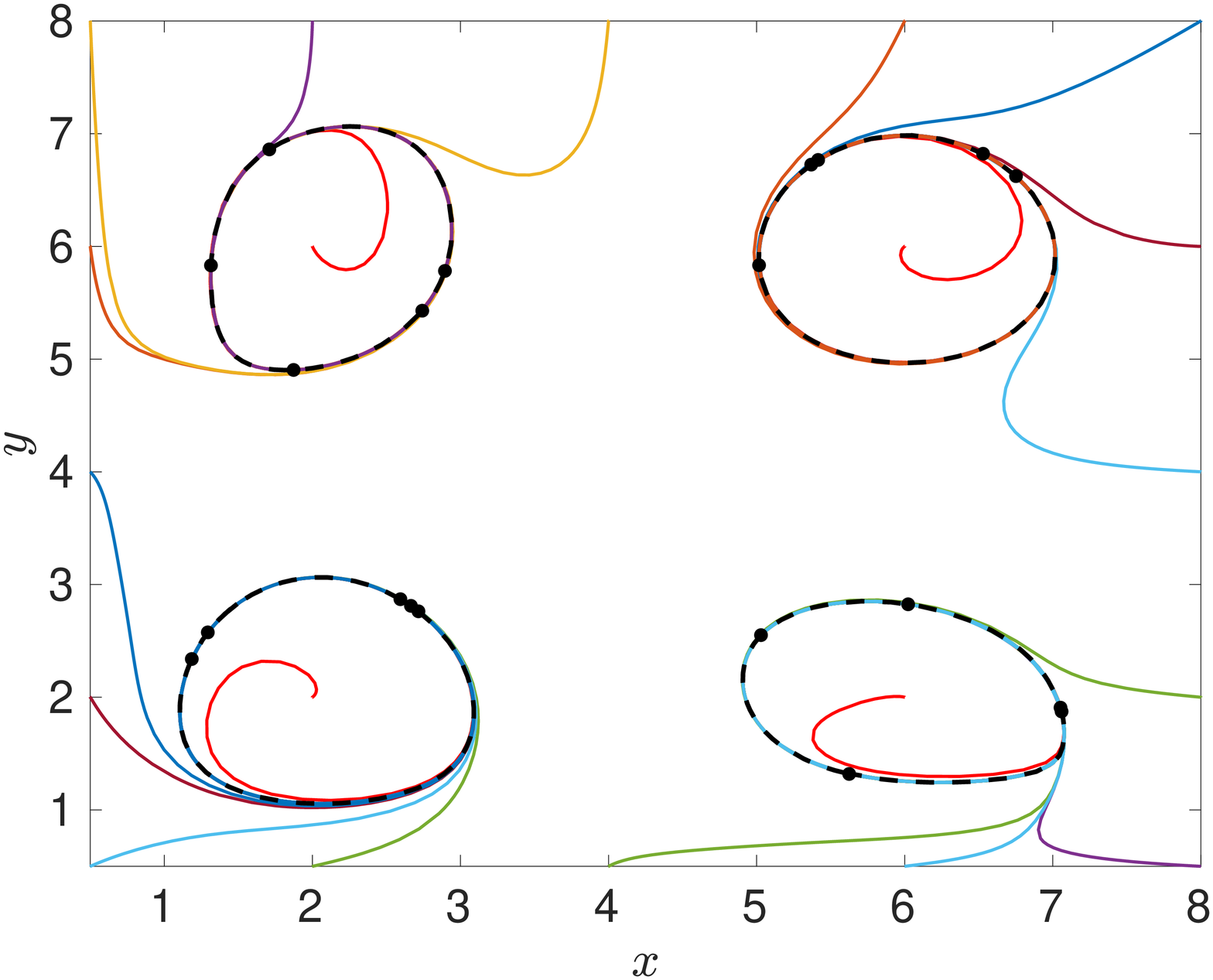}\;
\includegraphics[width=0.495\textwidth]{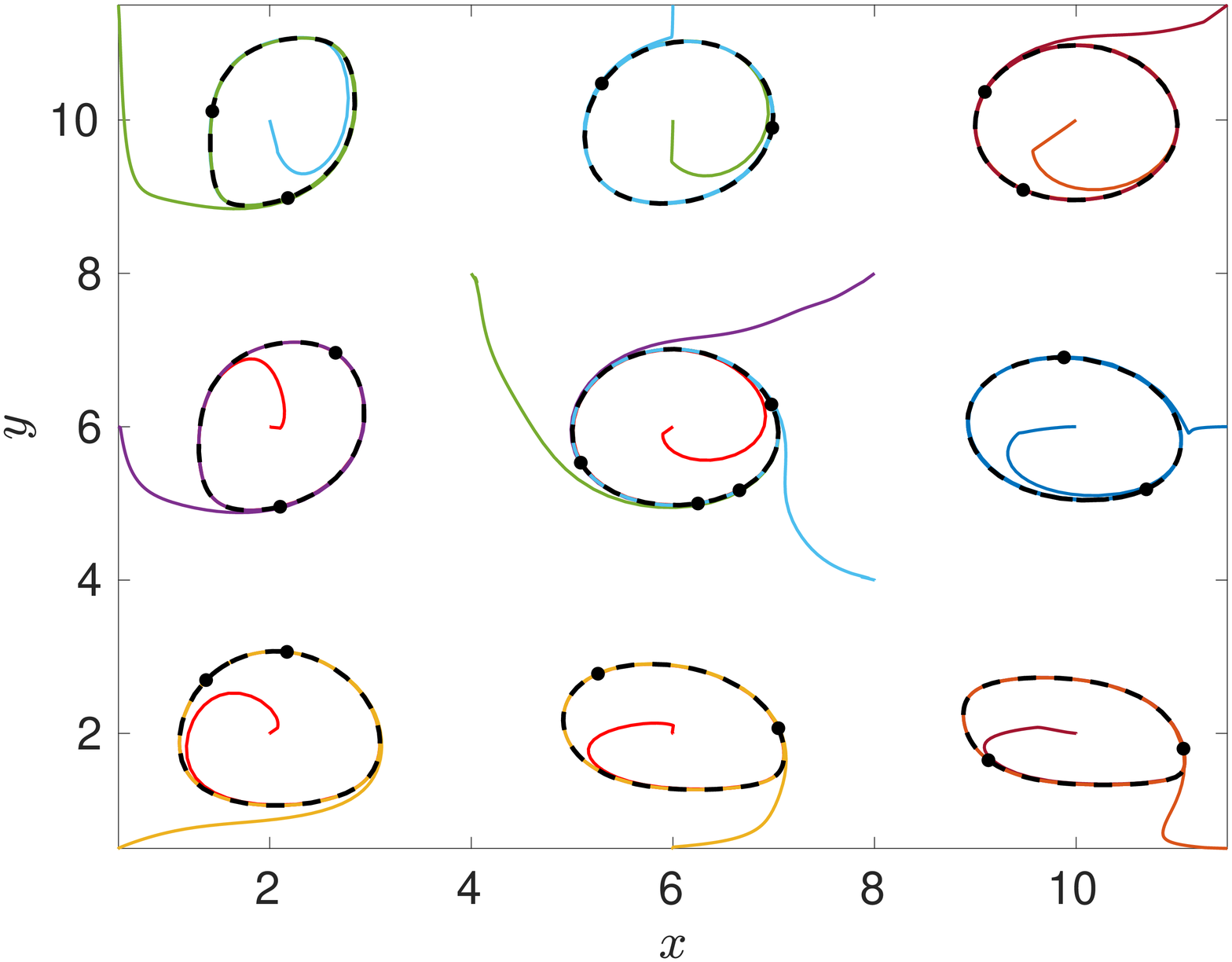}}
\caption{(a) {\it Twenty illustrative trajectories of the ODE system~$(\ref{xode5})$--$(\ref{viode5})$ for $K=4$, the parameter choices $a_1 = b_1 = a_2 = b_3 = 2$, $a_3 = b_2 = a_4 = b_4 = 6$, $\varepsilon=1$ and the initial condition~$(\ref{initcond2})$ with $c=1/2$. As $t \to \infty$, all trajectories approach one of the four limit cycles, which are plotted as the black dashed lines. As in Figure~$\ref{fig2}$, the black dots denote the final position of each calculated trajectory at time $t=100$.} \hfill\break (b) {\it Twenty illustrative trajectories of the ODE system~$(\ref{xode5})$--$(\ref{viode5})$ for $K=9$, the parameter choices $a_1 = b_1 = a_2 = b_3 = a_6 = b_7 = 2$, $a_3 = b_2 = a_4 = b_4 = a_5 = b_8 =6$, $a_7 = a_8 = a_9 = b_5 = b_6 =  b_9 = 10$, $\varepsilon=1$ and the initial condition~$(\ref{initcond3})$. As $t \to \infty$, all trajectories approach one of the nine limit cycles, which are plotted as the black dashed lines. The black dots denote the final position of each calculated trajectory at time $t=100$.}
}\label{fig3}
\end{figure}

In Figure~\ref{fig3}(b), we present illustrative dynamics of the ODE system~(\ref{xode5})--(\ref{viode5}) for $K=9$, showing that each computed trajectory converges to one of the 9 limit cycles denoted by black dashed lines. To illustrate that this behaviour does not require special choices of initial conditions, we used different initial conditions for $x(0)$ and $y(0)$ together with the initial conditions for variables $v_i$ satisfying 
\begin{equation}
v_i(0) = 1,
\qquad
\mbox{for} \; i=1,2,\dots,K.
\label{initcond3}
\end{equation}
However, a similar figure can be obtained if we replace (\ref{initcond3}) with the initial condition (\ref{initcond2}), or if we initialize all values of $v_i$, $i=1,2,\dots,K$ as zero (results not shown).

A CRN corresponding to reaction rate equations~(\ref{xode5})--(\ref{viode5}) can be obtained (by applying the construction in the proof of Lemma~\ref{lem1}) as the CRN with $K+2$ chemical species, i.e. using the notation in Definition~\ref{defcrn}, we have
\begin{equation}
\mathcal{S} = \left\{
X, Y, V_1, V_2, \dots, V_K
\right\}.
\label{setSchemsp}
\end{equation}
To specify the reaction complexes and chemical reactions, we expand the right-hand side of reaction rate equations~(\ref{xode5})--(\ref{viode5}). First, we rewrite ODEs~(\ref{viode5}) as
\begin{eqnarray}
\frac{\mbox{d}v_i}{\mbox{d}t}
\!&=&\!
- 
k_{i,1} \, v_i
+ 
k_{i,2} \, v_i x
+
k_{i,3} \, v_i y
-
k_{i,4} \, v_i x^2
-
k_{i,5} \, v_i y^2
+
k_{i,6} \, v_i x^3
+
k_{i,7} \, v_i y^3
\nonumber
\\
\!&-&\!
k_{i,8} \, v_i x^4
-
k_{i,9} \, v_i y^4
+
k_{i,10} \, v_i x^5
+
k_{i,11} \, v_i y^5
-
v_i x^6 / \varepsilon
-
v_i y^6 / \varepsilon
+
1 / \varepsilon,
\qquad
\label{expandedVi}
\end{eqnarray}
where $k_{i,j}$, $i=1,2,\dots,K,$ $j=1,2,\dots,11,$ are positive constants given by
\begin{equation*}
k_{i,1} = (1 + a_i^6 + b_i^6) / \varepsilon,
\quad
k_{i,2} = 6 a_i^5 / \varepsilon,
\quad
k_{i,3} = 6 b_i^5 / \varepsilon,
\quad
k_{i,4} = 15 a_i^4 / \varepsilon,
\end{equation*}
\begin{equation}
k_{i,5} = 15 b_i^4 / \varepsilon,
\quad
k_{i,6} = 20 a_i^3 / \varepsilon,
\quad
k_{i,7} = 20 b_i^3 / \varepsilon,
\quad
k_{i,8} = 15 a_i^2 / \varepsilon,
\;
\label{k1list1}
\end{equation}
\begin{equation*}
k_{i,9} = 15 b_i^2 / \varepsilon,
\quad
k_{i,10} = 6 a_i / \varepsilon
\quad
\mbox{and}
\quad
k_{i,11} = 6 b_i / \varepsilon.
\end{equation*}
Consequently, the right-hand side of equation~(\ref{viode5}) can be interpreted as the set of 14 chemical reactions for each $i=1,2,\dots,K$. We define it as 
\begin{eqnarray}
\mathcal{R}_i
\!&=&\!
\left\{
V_i \mathop{\longrightarrow}^{k_{i,1}} \emptyset, 
\quad
V_i+X \mathop{\longrightarrow}^{k_{i,2}} 2V_i+X,
\quad
V_i+Y \mathop{\longrightarrow}^{k_{i,3}} 2V_i+Y,
\quad
V_i+2X \mathop{\longrightarrow}^{k_{i,4}} 2X,
\right. 
\nonumber
\\
&&
\;\;
V_i+2Y \mathop{\longrightarrow}^{k_{i,5}} 2Y,
\qquad
V_i+3X \mathop{\longrightarrow}^{k_{i,6}} 2V_i+3X,
\qquad
V_i+3Y \mathop{\longrightarrow}^{k_{i,7}} 2V_i+3Y,
\qquad\;\;
\label{setri}
\\
&&
\;\;
V_i+4X \mathop{\longrightarrow}^{k_{i,8}} 4X,
\qquad
V_i+4Y \mathop{\longrightarrow}^{k_{i,9}} 4Y,
\qquad
V_i+5X \mathop{\longrightarrow}^{k_{i,10}} 2V_i+5X,
\nonumber
\\
&&
\left.
\;\;
V_i+5Y \mathop{\longrightarrow}^{k_{i,11}} 2V_i+5Y,
\quad
V_i+6X \mathop{\longrightarrow}^{1 / \varepsilon} 6X,
\quad
V_i+6Y \mathop{\longrightarrow}^{1 / \varepsilon} 6Y,
\quad
\emptyset \mathop{\longrightarrow}^{1 / \varepsilon} V_i
\right\}.
\nonumber
\end{eqnarray}
Consequently, reaction rate equations~(\ref{viode5}) correspond to $14\, K$ chemical reactions in sets $\mathcal{R}_i$, $i=1,2,\dots,K.$ Similarly, we rewrite
ODEs~(\ref{xode5})--(\ref{yode5}) as
\begin{eqnarray}
\frac{\mbox{d}x}{\mbox{d}t} 
\!&=&\! 
\sum_{i=1}^K  
\bigg[- v_i x^4
+ k_{i,12} \, v_i x^3
- k_{i,13} \, v_i x^2
+ k_{i,14} \, v_i x
+ a_i \, v_i x y^2
\nonumber
\\
&& \qquad
+ k_{i,15} \, v_i x^2 y
- k_{i,16} \, v_i x y
- \, v_i x^2 y^2\bigg], \quad\; 
\label{expandedx}
\\
\frac{\mbox{d}y}{\mbox{d}t} 
\!&=&\! 
\sum_{k=1}^K 
\bigg[- v_i y^4
+ k_{i,17} \, v_i y^3
- k_{i,18} \, v_i y^2
+ k_{i,19} \, v_i y
+ b_i \, v_i x^2 y
\nonumber
\\
&& \qquad
+ k_{i,20} \, v_i x y^2
- k_{i,21} \, v_i x y
- \, v_i x^2 y^2 \bigg], \quad\;
\label{expandedy}
\end{eqnarray}
where $k_{i,j}$, $i=1,2,\dots,K,$ $j=12,13,\dots,21,$ are constants given by
\begin{equation*}
k_{i,12} = 3 a_i,
\quad
k_{i,13} = 3 a_i^2 + b_i^2 - 1,
\quad
k_{i,14} = a_i^3 + a_i b_i^2 + b_i - a_i,
\quad
k_{i,15} = 2 b_i,
\end{equation*}
\begin{equation}
k_{i,16} = 1 + 2 a_i b_i,
\quad
k_{i,17} = 3 b_i,
\quad
k_{i,18} = a_i^2 + 3 b_i^2 - 1,
\label{kilist2}
\end{equation}
\begin{equation*}
k_{i,19} = b_i^3 + a_i^2 b_i - a_i - b_i,
\quad
k_{i,20} = 2 a_i,
\quad
k_{i,21} = 2 a_i b_i - 1.
\end{equation*}
Considering sufficiently large $a_i$ and $b_i$ (say, for $a_i>1$ and $b_i>1$), the constants~(\ref{kilist2}) are positive. Moreover,
since the term $- v_i x^2 y^2$ appears in both equations~(\ref{expandedx}) and~(\ref{expandedy}), the right-hand sides of equations~(\ref{xode5})--(\ref{yode5}) can be interpreted as the set of $15 \, K$ chemical reactions. We define 
\begin{eqnarray}
\mathcal{R}_i^*
\!&=&\!
\left\{
V_i + 4X \mathop{\longrightarrow}^{1} V_i + 3X, 
\quad
V_i + 3X \mathop{\longrightarrow}^{k_{i,12}} V_i + 4X,
\quad
V_i + 2X \mathop{\longrightarrow}^{k_{i,13}} V_i + X,
\right. 
\nonumber
\\
&&
\;\;
V_i + X \mathop{\longrightarrow}^{k_{i,14}} V_i + 2X,
\qquad
V_i + X + 2Y \mathop{\longrightarrow}^{a_i} V_i + 2X + 2Y,
\nonumber
\\
&&
\;\;
V_i + 2X + Y \mathop{\longrightarrow}^{k_{i,15}} V_i + 3X +Y,
\qquad
V_i + X + Y \mathop{\longrightarrow}^{k_{i,16}} V_i + Y,
\nonumber
\\
&&
\;\;
V_i + 2X + 2Y \mathop{\longrightarrow}^{1} V_i + X + Y,
\qquad
V_i + 4Y \mathop{\longrightarrow}^{1} V_i + 3Y,
\nonumber
\\
&&
\;\;
V_i + 3Y \mathop{\longrightarrow}^{k_{i,17}} V_i + 4Y,
\quad
V_i + 2Y \mathop{\longrightarrow}^{k_{i,18}} V_i + Y,
\quad
V_i + Y \mathop{\longrightarrow}^{k_{i,19}} V_i + 2Y,
\nonumber
\\
&&
\;\;
V_i + 2X + Y \mathop{\longrightarrow}^{b_i} V_i + 2X + 2Y,
\qquad
V_i + X + 2Y \mathop{\longrightarrow}^{k_{i,20}} V_i + X + 3Y,
\nonumber
\\
&&
\left.
\;\;
V_i + X + Y \mathop{\longrightarrow}^{k_{i,21}} V_i + X
\right\},
\qquad
\mbox{for} \quad i=1,2,\dots,K.
\label{setristar}
\end{eqnarray}
Then, we conclude that the reaction rate equations~(\ref{xode5})--(\ref{viode5}) correspond to the CRN with $N=K+2$ chemical species and $29 \, K$ chemical reactions of at most seventh order given by
\begin{equation}
\mathcal{R}
=
\bigcup_{i=1}^K \mathcal{R}_i \cup \mathcal{R}_i^*.
\label{deduced_reactions} 
\end{equation}
The CRN $(\mathcal{S},\mathcal{C},\mathcal{R})$ consisting of chemical species~$\mathcal{S}$ given by~(\ref{setSchemsp}) and chemical reactions~$\mathcal{R}$ given by~(\ref{deduced_reactions}) is the CRN which we will use to prove Theorem~\ref{thmcrn1} in Section~\ref{sec6}. The corresponding set of reaction complexes~$\mathcal{C}$ can be inferred from the provided lists of reactions~$\mathcal{R}_i$ and $\mathcal{R}_i^*$, $i=1,2,\dots,K$, given by~(\ref{setri}) and~(\ref{setristar}).

\section{Proof of Theorem~\ref{thmcrn1}}
\label{sec6}

The idea of the proof of Theorem~\ref{thmcrn1} is similar to the one chosen in Sections~\ref{sec3} and~\ref{sec4}, where we have first proved Lemma~\ref{lem3} about the existence of $K$ limit cycles in the planar ODE system~(\ref{xode2})--(\ref{yode2}) and then we have used it to prove the existence of $K$ limit cycles in the $(K+2)$-dimensional ODE system in Lemma~\ref{lem4}. In this section, we will again start by formulating
Lemma~\ref{lem5} for a planar ODE system, which we will use in Lemma~\ref{lem6} to prove Theorem~\ref{thmcrn1} considering the $(K+2)$-dimensional ODE system~(\ref{xode5})--(\ref{viode5}).
The planar ODE system is derived by applying the $x$-factorable transformation to the planar ODE system~(\ref{xode2})--(\ref{yode2}). We obtain
\begin{eqnarray}
\frac{\mbox{d}x}{\mbox{d}t} 
\!&=&\! 
\sum_{k=1}^K 
x\,\,f_k(x-a_k,y-b_k) 
\!=\! 
x\,\,f(x,y)
\,, \quad\; 
\label{xode6}\\
\frac{\mbox{d}y}{\mbox{d}t} 
\!&=&\! 
\sum_{k=1}^K y\,\,g_k(x-a_k,y-b_k)
\!=\! 
y\,\,g(x,y)
\,, \quad\;
\label{yode6}
\end{eqnarray}
where we have used the notation $f_k(\cdot,\cdot)$ and $g_k(\cdot,\cdot)$ introduced in equations~(\ref{f_k}), (\ref{g_k}) and (\ref{fgdef}). 

The dynamics of the ODE system~(\ref{xode6})--(\ref{yode6}) is similar to the dynamics of the original planar ODE system~(\ref{xode2})--(\ref{yode2}) in the same way as the dynamics of the $(K+2)$-dimensional extended  ODE system~(\ref{xode5})--(\ref{viode5}) is similar to the dynamics of the $(K+2)$-dimensional extended ODE system~(\ref{xode4})--(\ref{viode4}).
We have already observed in Figure~\ref{fig3}(a) that the limit cycle around the point $(a_i,b_i)=(6,6)$ of the ODE system~(\ref{xode4})--(\ref{viode4}) is relatively circular. On the other hand, the shape of the
limit cycles can more significantly differ between Figures~\ref{fig2}(b) and~\ref{fig3}(a) if the corresponding parameters $a_i$ and $b_i$ are not equal to each other. Motivated by this observation, we will study the case $a_i = b_i$ in Lemma~\ref{lem5} and prove that it is possible to choose these parameters in a way that the planar ODE system~(\ref{xode6})--(\ref{yode6}) has (at least) $K$ stable limit cycles. This result is sufficient for the proof of Theorem~\ref{thmcrn1}. However, we also note that the existence of limit cycles of the ODE system~(\ref{xode6})--(\ref{yode6}) is not restricted to the case $a_i = b_i$ and a more general lemma could be stated and proven, as we did in Lemma~\ref{lem3} where the existence of $K$ limit cycles has been proven under a relatively general condition~(\ref{sepass}). The advantage of the case $a_i = b_i$ is that it simplifies the proof of Lemma~\ref{lem5}, because we can use the approach and notations introduced in the proof of Lemma~\ref{lem3}.

\begin{lemma}
\label{lem5}
Let us assume that 
\begin{equation}
a_i= b_i=8 {\hskip 0.2mm} i {\hskip 0.15mm} K
\label{lem5ass}
\end{equation}
for
$i=1,2,\dots,K.$
Then the ODE system~$(\ref{xode6})$--$(\ref{yode6})$ has at least $K$ stable limit cycles.
\end{lemma}

\begin{proof}
Let us define regions $\Omega_i \subset {\mathbb R}^2$, $i=1,2,\dots,K,$ together with their boundary parts $\partial \Omega_{i1}$ and $\partial \Omega_{i2}$ by (\ref{omegaidef}), (\ref{boundary1}) and (\ref{boundary2}). Our choice of values of $a_i$ and $b_i$ in (\ref{lem5ass}) satisfies the assumption~(\ref{sepass}) in Lemma \ref{lem3}. Therefore, the ODE system~(\ref{xode2})--(\ref{yode2}) has with parameters given by~(\ref{lem5ass}) at least $K$ stable limit cycles. Moreover, in the proof of Lemma~\ref{lem3}, we have shown that each region $\Omega_i$ does not include any equilibrium of the planar ODE system~(\ref{xode2})--(\ref{yode2}).
Any equilibrium of the ODE system~(\ref{xode6})--(\ref{yode6}) is either located on the $x$-axis or $y$-axis, or it is also an equilibrium of the ODE system~(\ref{xode2})--(\ref{yode2}). 
However, our assumption~(\ref{lem5ass}) implies that no region $\Omega_i$, $i=1,2,\dots,K,$ intersects with the $x$-axis or $y$-axis. Therefore, we conclude that each $\Omega_i$, for $i=1,2,\dots,K$, does not contain any equilibrium of the ODE system~(\ref{xode6})--(\ref{yode6}).  

Next, consider any point $(x_b,y_b)\in\partial\Omega_i$. We will compute the scalar product of vectors 
\begin{equation}
(x_b-a_i,y_b-b_i)
\qquad
\mbox{and}
\qquad
\big(x_b\,f(x_b,y_b), y_b\,g(x_b,y_b)\big)
\label{scalar_vectors2}
\end{equation}
by rewriting the second vector as
a sum of two vectors
\begin{equation}
\big(x_b\,f(x_b,y_b), y_b\,g(x_b,y_b)\big)
=
x_b \big(f(x_b,y_b),g(x_b,y_b)\big)
+
\big(0, (y_b-x_b) \, g(x_b,y_b) \big).
\label{revsecvec}
\end{equation}
The scalar product of vectors
\begin{equation}
(x_b-a_i,y_b-b_i)
\qquad
\mbox{and}
\qquad
x_b \big(f(x_b,y_b), g(x_b,y_b)\big)
\label{scalar_vectors3}
\end{equation}
has already been calculated in the proof of Lemma~\ref{lem3} starting with equation~(\ref{scalar_vectors}). We obtained that it is negative for $(x_b,y_b)\in\partial\Omega_{i1}$ and positive for $(x_b,y_b)\in\partial\Omega_{i2}$. Therefore, the vector~$x_b \big(f(x_b,y_b), g(x_b,y_b)\big)$ always points inside the domain $\Omega_i$ on all parts of the boundary $\partial \Omega_i.$ Next, we want to show that this conclusion also holds if vector~$x_b \big(f(x_b,y_b), g(x_b,y_b)\big)$ is modified by adding the vector~$\big(0, (y_b-x_b) \, g(x_b,y_b) \big)$  as it is done in equation~(\ref{revsecvec}). To do this, we note that our choice of parameters~(\ref{lem5ass}) implies that
$$
(a_i-a_j)^2+(b_i-b_j)^2 
=
128 {\hskip 0.2mm} (i-j)^2 K^2
$$
for all $i,j=1,2,\dots,K$, which not only satisfies the assumption~(\ref{sepass})
but it can be used in equation~(\ref{destimates}) to make a stronger conclusion that the scalar product of vectors~(\ref{scalar_vectors3}) is at most~$-1.45$ for $(x_b,y_b)\in\partial\Omega_{i1}$ and at least $1.45$ for $(x_b,y_b)\in\partial\Omega_{i2}$. Thus, we only need to show that the scalar product of vectors
\begin{equation}
(x_b-a_i,y_b-b_i)
\qquad
\mbox{and}
\qquad
\big(0, (y_b-x_b) \, g(x_b,y_b) \big)
\label{scalar_vectors4}
\end{equation}
is in absolute value less than $1.45$ to conclude that the original scalar product~(\ref{scalar_vectors2}) is negative for $(x_b,y_b)\in\partial\Omega_{i1}$ and positive for $(x_b,y_b)\in\partial\Omega_{i2}$. 
Using the definition of $g(\cdot,\cdot)$ in~(\ref{fgdef}) and the notation~$z_1=x_b-a_i$, $z_2=y_b-b_i$ introduced in the proof of Lemma~\ref{lem3}, we have $y_b-x_b=z_2-z_1$ and
the scalar product~(\ref{scalar_vectors4}) can be written as
\begin{equation}
(z_2-z_1)\,z_2\,\,
g_i(z_1,z_2)
+
(z_2-z_1) \, z_2 \!\!\!\!
\sum_{k=1, k \ne i}^K
\!\!
g_k(x_b-a_k,y_b-a_k).
\label{auxscalarprod}
\end{equation}
Since we have
$$
\max_{(x_b,y_b)\in\partial\Omega_i}
\big\vert
(z_2-z_1)\,z_2\,\,
g_i(z_1,z_2)
\big\vert
=
\max_{z_1^2+z_2^2=2 \; 
(\mbox{{\scriptsize or}} \;
1/2)}
\big\vert
(z_2-z_1)\,z_2\,\,
g_i(z_1,z_2)
\big\vert
\le
1.4
$$
and the second term in~(\ref{auxscalarprod}) is also less than $0.05$, we conclude that the scalar product~(\ref{scalar_vectors2}) is negative for $(x_b,y_b)\in\partial\Omega_{i1}$ and positive for $(x_b,y_b)\in\partial\Omega_{i2}$. Therefore, the vector~$\big(x_b\,f(x_b,y_b), y_b\,g(x_b,y_b)\big)$ always points inside the domain $\Omega_i$ on all parts of the boundary $\partial \Omega_i.$ 
In particular, applying Poincar\'e-Bendixson theorem, we conclude that each $\Omega_i$ contains at least one stable limit cycle. Since $\Omega_i$, $i=1,2,\dots,K,$ are pairwise disjoint, this implies that the ODE system~(\ref{xode6})--(\ref{yode6}) has at least $K$ stable limit cycles.
\end{proof}

\begin{lemma}
\label{lem6}
Let us assume that constants $a_i,$ $b_i$, $i=1,2,\dots,K$ are given by~$(\ref{lem5ass})$.
Then there exists $\varepsilon_0>0$ such that the reaction rate equations~$(\ref{xode5})$--$(\ref{viode5})$ have at least $K$ stable limit cycles for all $\varepsilon \in (0,\varepsilon_0).$
\end{lemma}

\begin{proof}
This follows directly from Lemma~\ref{lem5} and Tikhonov's theorem~\cite{Tikhonov:1952:SDE,Klonowski:1983:SPC}.
\end{proof}
 
\noindent
The existence of $K$ limit cycles in the CRN~(\ref{deduced_reactions}) follows by application of Lemma~\ref{lem6}. The chemical system~(\ref{deduced_reactions}) has $(K+2)$ chemical species $X,$ $Y,$ $V_1$, $V_2$, \dots, $V_K$, which are subject to $29K$ chemical reactions, so, by construction, we also establish
bounds in part (ii) of Theorem~\ref{thmcrn1} on $N(K)$ and $M(K)$. This concludes the proof of Theorem~\ref{thmcrn1}.

\section{Proof of Theorem~\ref{thmcrn2}}
\label{sec7}

In Theorem~\ref{thmcrn1}, we have established that the reaction rate equations~(\ref{xode5})--(\ref{viode5}) describing the CRN~(\ref{deduced_reactions}) have at least $K$ stable limit cycles. Since the right-hand sides of ODEs~(\ref{xode5})--(\ref{viode5}) include polynomials up to the order 7, the resulting chemical reactions~(\ref{deduced_reactions}) are reactions of the order at most 7. However, in practice, every higher-order reactions can be subdivided into elementary steps, which are at most bimolecular (second order). Therefore, we focus here on the proof of Theorem~\ref{thmcrn2} which restricts our considerations to at most second-order kinetics. We prove it by further extending the number of variables in the reaction rate equations~(\ref{xode5})--(\ref{viode5}), i.e. by adding intermediary chemical species and elementary reactions into the CRN~(\ref{deduced_reactions}). The resulting CRN has $N=7K+14$ chemical species denoted by
\begin{equation}
\mathcal{S} = \big\{
X, Y, W_1, W_2, \dots, W_{12}
\big\}
\cup
\bigcup_{i=1}^K 
\left\{
V_i,
Z_{i,1},Z_{i,2},Z_{i,3},Z_{i,4},Z_{i,5},Z_{i,6} 
\right\},
\label{species2ndorder}    
\end{equation}
where we use the notation introduced in Definition~\ref{defcrn} of CRNs. The concentrations
$x,$ $y,$ $v_i$, $w_1$, $w_2$, $\dots$, $w_{12}$, $z_{i,j}$ for $i=1,2,\dots,K$ and $j=1,2,\dots,6$ evolve according to reaction rate equations
\begin{eqnarray}
\frac{\mbox{d}x}{\mbox{d}t} 
\!&=&\! 
\sum_{i=1}^K  
\bigg[- x z_{i,3}
+ k_{i,12} \, v_i w_2
- k_{i,13} \, x z_{i,1}
+ k_{i,14} \, v_i x
+ a_i \, v_i w_{11}
\nonumber
\\
&& \qquad
+ k_{i,15} \, v_i w_{12}
- k_{i,16} \, x z_{i,2}
- \, x z_{i,5} \bigg], \quad\; 
\label{expandedx2}
\\
\frac{\mbox{d}y}{\mbox{d}t} 
\!&=&\! 
\sum_{k=1}^K 
\bigg[- y z_{i,4}
+ k_{i,17} \, v_i w_7
- k_{i,18} \, y z_{i,2}
+ k_{i,19} \, v_i y
+ b_i \, v_i w_{12}
\nonumber
\\
&& \qquad
+ k_{i,20} \, v_i w_{11}
- k_{i,21} \, y z_{i,1}
- \, y z_{i,6} \bigg], \quad\;
\label{expandedy2}
\\
\frac{\mbox{d}v_i}{\mbox{d}t}
\!&=&\! 
-
k_{i,1} \, v_i
+ 
k_{i,2} \, v_i x
+
k_{i,3} \, v_i y
-
k_{i,4} \, v_i w_1
-
k_{i,5} \, v_i w_6
\nonumber
\\
&& \quad 
+ \,
k_{i,6} \, v_i w_2
+
k_{i,7} \, v_i w_7
-
k_{i,8} \, v_i w_3
-
k_{i,9} \, v_i w_8
\label{expandedVi2}
\\
&& \quad
+ \,
k_{i,10} \, v_i w_4
+
k_{i,11} \, v_i w_9
-
v_i w_5 / \varepsilon
-
v_i w_{10} / \varepsilon
+
1 / \varepsilon,
\nonumber
\\
\delta \,
\frac{\mbox{d}w_1}{\mbox{d}t}
\!&=&\!
x^2 - w_1,
\qquad
\delta \,
\frac{\mbox{d}w_2}{\mbox{d}t}
=
x w_1 - w_2,
\qquad
\delta \,
\frac{\mbox{d}w_3}{\mbox{d}t}
=
x w_2 - w_3,
\label{w123eq}
\\
\delta \,
\frac{\mbox{d}w_4}{\mbox{d}t}
\!&=&\!
x w_3 - w_4,
\qquad
\delta \,
\frac{\mbox{d}w_5}{\mbox{d}t}
=
x w_4 - w_5,
\qquad
\delta \,
\frac{\mbox{d}w_6}{\mbox{d}t}
=
y^2 - w_6,
\label{w456eq}
\\
\delta \,
\frac{\mbox{d}w_7}{\mbox{d}t}
\!&=&\!
y w_6 - w_7,
\qquad
\delta \,
\frac{\mbox{d}w_8}{\mbox{d}t}
=
y w_7 - w_8,
\qquad
\delta \,
\frac{\mbox{d}w_9}{\mbox{d}t}
=
y w_8 - w_9,
\label{w789eq}
\\
\delta \,
\frac{\mbox{d}w_{10}}{\mbox{d}t}
\!&=&\!
y w_9 - w_{10},
\quad\;
\delta \,
\frac{\mbox{d}w_{11}}{\mbox{d}t}
=
x w_6 - w_{11},
\quad\;
\delta \,
\frac{\mbox{d}w_{12}}{\mbox{d}t}
=
y w_1 - w_{12},
\label{w101112eq}
\\
\delta \,
\frac{\mbox{d}z_{i,1}}{\mbox{d}t}
\!&=&\!
v_i x - z_{i,1},
\qquad
\delta \,
\frac{\mbox{d} z_{i,2}}{\mbox{d}t}
=
v_i y - z_{i,2},
\qquad
\delta \,
\frac{\mbox{d} z_{i,3}}{\mbox{d}t}
=
v_i w_2 - z_{i,3},
\label{zi1i2i3eq}
\\
\delta \,
\frac{\mbox{d} z_{i,4}}{\mbox{d}t}
\!&=&\!
v_i w_7 - z_{i,4},
\quad
\delta \,
\frac{\mbox{d} z_{i,5}}{\mbox{d}t}
=
v_i w_{11} - z_{i,5},
\quad
\delta \,
\frac{\mbox{d} z_{i,6}}{\mbox{d}t}
=
v_i w_{12} - z_{i,6},
\quad
\label{zi4i5i6eq}
\end{eqnarray}
where $\delta>0$, $\varepsilon>0$ and $k_{i,j}$, $i=1,2,\dots,K,$ $j=1,2,\dots,21,$ are positive constants given by~(\ref{k1list1}) and (\ref{kilist2}).
Considering the limit $\delta \to 0$ in equations~(\ref{w123eq})--(\ref{zi4i5i6eq}), we obtain
$$
w_1 = x^2, \quad
w_2 = x^3, \quad
w_3 = x^4, \quad
w_4 = x^5, \quad
w_5 = x^6, \quad
w_6 = y^2, 
$$
\begin{equation}
w_7 = y^3, \quad
w_8 = y^4, \quad
w_9 = y^5, \quad
w_{10} = y^6, \quad
w_{11} = x y^2, \quad
w_{12} = x^2 y,
\label{wvallimdel}
\end{equation}
$$
z_{i,1} = v_i x, \quad
z_{i,2} = v_i y, \quad
z_{i,3} = v_i x^3 \! , \quad
z_{i,4} = v_i y^3 \! , \quad
z_{i,5} = v_i x y^2 \! , \quad
z_{i,6} = v_i x^2 y.
$$
Substituting the limiting values~(\ref{wvallimdel}) for $w_\ell$ and $z_{i,j}$, $\ell=1,2,\dots,12,$ $i=1,2\dots,K$, $j=1,2,\dots,6$, into equations~(\ref{expandedx2})--(\ref{expandedVi2}), we obtain equations~(\ref{expandedx}),
(\ref{expandedy}) and~(\ref{expandedVi}), which are equal to the reaction rate equations~(\ref{xode5})--(\ref{viode5}). In particular, we deduce the following lemma.

\begin{lemma}
\label{lem7}
Let us assume that constants $a_i,$ $b_i$, $i=1,2,\dots,K$ are given by~$(\ref{lem5ass})$.
Then there exist $\delta_0>0$ and $\varepsilon_0>0$ such that the reaction rate equations~$(\ref{expandedx2})$--$(\ref{zi4i5i6eq})$ have at least $K$ stable limit cycles for all $\delta \in (0,\delta_0)$ and $\varepsilon \in (0,\varepsilon_0).$
\end{lemma}

\begin{proof}
This follows directly from Lemma~\ref{lem6} and Tikhonov's theorem~\cite{Tikhonov:1952:SDE,Klonowski:1983:SPC}.
\end{proof}

\noindent
The right-hand sides of reaction rate equations~(\ref{expandedx2})--(\ref{zi4i5i6eq}) only include quadratic terms. Therefore, there exists a CRN corresponding to reaction rate equations~(\ref{expandedx2})--(\ref{zi4i5i6eq}) which includes (at most) second-order reactions. We can obtain it by applying the construction in the proof of Lemma~\ref{lem1}. The right-hand sides of equations (\ref{expandedx2}) and (\ref{expandedy2}) can be interpreted as the set of $16 \, K$ chemical reactions
(compare with~(\ref{setristar}) for ODEs~(\ref{xode5})--(\ref{yode5}))
\begin{eqnarray}
\mathcal{R}_i^{s,*}
\!&=&\!
\left\{
X + Z_{i,3} \mathop{\longrightarrow}^{1} Z_{i,3}, 
\quad
V_i + W_2 \mathop{\longrightarrow}^{k_{i,12}} V_i + W_2 + X,
\quad
X + Z_{i,1} \mathop{\longrightarrow}^{k_{i,13}} Z_{i,1},
\right. 
\nonumber
\\
&&
\;\;
V_i + X \mathop{\longrightarrow}^{k_{i,14}} V_i + 2X,
\qquad
V_i + W_{11} \mathop{\longrightarrow}^{a_i} V_i + W_{11} + X,
\nonumber
\\
&&
\;\;
V_i + W_{12} \mathop{\longrightarrow}^{k_{i,15}} V_i + W_{12} + X,
\qquad
X + Z_{i,2} \mathop{\longrightarrow}^{k_{i,16}} Z_{i,2},
\nonumber
\\
&&
\;\;
X + Z_{i,5} \mathop{\longrightarrow}^{1} Z_{i,5},
\qquad
Y + Z_{i,6} \mathop{\longrightarrow}^{1} Z_{i,6},
\qquad
Y + Z_{i,4} \mathop{\longrightarrow}^{1} Z_{i,4},
\nonumber
\\
&&
\;\;
V_i + W_{7} \mathop{\longrightarrow}^{k_{i,17}} V_i + W_7 + Y,
\quad
Y + Z_{i,2} \mathop{\longrightarrow}^{k_{i,18}} Z_{i,2},
\quad
V_i + Y \mathop{\longrightarrow}^{k_{i,19}} V_i + 2Y,
\nonumber
\\
&&
\;\;
V_i + W_{12} \mathop{\longrightarrow}^{b_i} V_i + W_{12} + Y,
\qquad
V_i + W_{11} \mathop{\longrightarrow}^{k_{i,20}} V_i + W_{11} + Y,
\nonumber
\\
&&
\left.
\;\;
Y + Z_{i,1} \mathop{\longrightarrow}^{k_{i,21}} Z_{i,1}
\right\},
\qquad
\mbox{for} \quad i=1,2,\dots,K.
\label{setristarsecond}
\end{eqnarray}
The right-hand side of equation~(\ref{expandedVi2}) can be interpreted as the set of $14$ chemical reactions for each $i=1,2,\dots,K$ (compare with~(\ref{setri}) for the right-hand side of ODE~(\ref{viode5}))
\begin{eqnarray}
\mathcal{R}_i^{s}
\!&=&\!
\left\{
V_i \mathop{\longrightarrow}^{k_{i,1}} \emptyset, 
\quad
V_i+X \mathop{\longrightarrow}^{k_{i,2}} 2V_i+X,
\quad
V_i+Y \mathop{\longrightarrow}^{k_{i,3}} 2V_i+Y,
\right.
\nonumber
\\
&&
\quad
V_i+W_1 \mathop{\longrightarrow}^{k_{i,4}} W_1,
\qquad
V_i+W_6 \mathop{\longrightarrow}^{k_{i,5}} W_6,
\qquad
V_i+W_2 \mathop{\longrightarrow}^{k_{i,6}} 2V_i+W_2,
\nonumber
\\
&&
\quad
V_i+W_7 \mathop{\longrightarrow}^{k_{i,7}} 2V_i+W_7,
\qquad
V_i+W_3 \mathop{\longrightarrow}^{k_{i,8}} W_3,
\qquad
V_i+W_8 \mathop{\longrightarrow}^{k_{i,9}} W_8,
\nonumber
\\
&&
\quad
V_i+W_4 \mathop{\longrightarrow}^{k_{i,10}} 2V_i+W_4,
\qquad
V_i+W_9 \mathop{\longrightarrow}^{k_{i,11}} 2V_i+W_9,
\nonumber
\\
&&
\left.
\quad\!
V_i+W_5 \mathop{\longrightarrow}^{1 / \varepsilon} W_5,
\qquad
V_i+W_{10} \mathop{\longrightarrow}^{1 / \varepsilon} W_{10},
\qquad
\emptyset \mathop{\longrightarrow}^{1 / \varepsilon} V_i
\right\}.
\label{setrisecond}
\end{eqnarray}
Consequently, reaction rate equations~(\ref{expandedx2})--(\ref{expandedVi2}) correspond to $30\, K$ chemical reactions in sets $\mathcal{R}_i^{s,*}$ and $\mathcal{R}_i^s$, $i=1,2,\dots,K.$ This is already more that $29 \, K$ chemical reactions used in Theorem~\ref{thmcrn1}, because we did not combine two terms on the right-hand sides into one reaction as we did in the set $\mathcal{R}_i^{*}$ (this is further discussed in equation~(\ref{vi2x2y}) in
Section~\ref{sec9}). Moreover, there are additional chemical reactions corresponding to the dynamics of additional chemical species in equations~(\ref{w123eq})--(\ref{zi4i5i6eq}).
The right-hand sides of equations~(\ref{w123eq})--(\ref{w101112eq}) can be interpreted as the set of $24$ chemical reactions given as
\begin{eqnarray}
\mathcal{R}^{w}
\!&=&\!
\left\{
2 X \mathop{\longrightarrow}^{1/\delta} 2X + W_1, 
\quad
2 Y \mathop{\longrightarrow}^{1/\delta} 2 Y + W_6,
\right.
\nonumber
\\
&&
\quad
X + W_j \mathop{\longrightarrow}^{1/\delta} X + W_j + W_{j+1}, \quad \mbox{for} \;\; j=1,2,3,4, 
\nonumber
\\
&&
\quad
Y + W_j \mathop{\longrightarrow}^{1/\delta} Y + W_j + W_{j+1}, \quad \mbox{for} \;\; j=6,7,8,9, 
\nonumber
\\
&&
\quad
X + W_6 \mathop{\longrightarrow}^{1/\delta} X + W_6 + W_{11}, 
\quad
Y + W_1 \mathop{\longrightarrow}^{1/\delta} Y + W_1 +
W_{12},
\nonumber
\\
&&
\left.
\quad
W_\ell \mathop{\longrightarrow}^{1/\delta} \emptyset,
\quad \mbox{for} \;\; \ell = 1,2,\dots,12
\right\}.
\label{setrw}
\end{eqnarray}
Finally, the right-hand sides of equations~(\ref{zi1i2i3eq})-(\ref{zi4i5i6eq}) can be interpreted as the set of $12$ chemical reactions for each $i=1,2,\dots,K$ given by
\begin{eqnarray}
\mathcal{R}_i^{z}
\!&=&\!
\left\{
X + V_i \mathop{\longrightarrow}^{1/\delta} X + V_i + Z_{i,1}, 
\quad
Y + V_i \mathop{\longrightarrow}^{1/\delta} Y + V_i + Z_{i,2},
\right.
\nonumber
\\
&&
\quad
V_i + W_2 \mathop{\longrightarrow}^{1/\delta} V_i + W_2 + Z_{i,3}, 
\quad
V_i + W_7 \mathop{\longrightarrow}^{1/\delta} V_i + W_7 + Z_{i,4},
\nonumber
\\
&&
\quad
V_i + W_{11} \mathop{\longrightarrow}^{1/\delta} V_i + W_{11} + Z_{i,5}, 
\quad
V_i + W_{12} \mathop{\longrightarrow}^{1/\delta} V_i + W_{12} + Z_{i,6},
\nonumber
\\
&&
\left.
\quad
Z_{i,j} \mathop{\longrightarrow}^{1/\delta} \emptyset,
\quad \mbox{for} \;\; j = 1,2,\dots,6
\right\}.
\label{setriz}
\end{eqnarray}
In summary, we conclude that the reaction rate equations~(\ref{expandedx2})--(\ref{zi4i5i6eq}) correspond to the CRN with $N=7K+14$ chemical species and $42 \, K + 24$ chemical reactions given by
\begin{equation}
\mathcal{R}
=
\mathcal{R}^w
\cup
\bigcup_{i=1}^K \mathcal{R}_i^s \cup \mathcal{R}_i^{s,*} \cup \mathcal{R}_i^z.
\label{deduced_reactionssecond} 
\end{equation}
Using Lemma~\ref{lem7}, we deduce that
the CRN $(\mathcal{S},\mathcal{C},\mathcal{R})$ consisting of chemical species~$\mathcal{S}$ given by~(\ref{species2ndorder}) and chemical reactions~$\mathcal{R}$ given by~(\ref{deduced_reactionssecond}) is an example of a CRN which satisfies Theorem~\ref{thmcrn2}. The corresponding set of reaction complexes~$\mathcal{C}$ can be inferred from the provided lists of reactions~$\mathcal{R}_i^{s,*}$, $\mathcal{R}_i^s$, $\mathcal{R}^w$ and $\mathcal{R}_i^z$, for $i=1,2,\dots,K$, given by~(\ref{setristarsecond}), (\ref{setrisecond}) (\ref{setrw}) and~(\ref{setriz}).

\section{Proof of Theorem~\ref{thmcrn3}}
\label{sec8}

Given an arbitrarily large integer $K \in {\mathbb N}$, we will show that there exists a CRN with two chemical species such that its reaction rate equations have at least $K$ stable limit cycles and the order of the chemical reactions is at most $n(K)=6K-2$. To do that, we start with the planar ODEs~(\ref{xode2})--(\ref{yode2}) and renormalize time $t$ to get a planar system with polynomial ODEs. Using an auxiliar function
$$
h(x,y) = \prod_{k=1}^K \Big( 1 + (x-a_k)^6 + (y-b_k)^6 \Big),
$$
we define our new time variable $\tau$ by
$$
\tau = \int_0^t \frac{1}{h(x(s),y(s))} \, \mbox{d} s.
$$
Then we obtain
\begin{eqnarray}
\frac{\mbox{d}x}{\mbox{d}\tau} 
\!&=&\! 
\frac{\mbox{d}x}{\mbox{d}t} 
\frac{\mbox{d}t}{\mbox{d}\tau}
=
h(x,y) 
\sum_{k=1}^K 
\frac{(x-a_k)
\big\{1-(x-a_k)^2-(y-b_k)^2\big\}
-(y-b_k)}{
1 + (x-a_k)^6 + (y-b_k)^6} 
\,, 
\qquad
\label{xode2t}\\
\frac{\mbox{d}y}{\mbox{d}\tau} 
\!&=&\! 
\frac{\mbox{d}y}{\mbox{d}t} 
\frac{\mbox{d}t}{\mbox{d}\tau}
=
h(x,y) 
\sum_{k=1}^K 
\frac{(y-b_k)
\big\{1-(x-a_k)^2-(y-b_k)^2\big\}
+(x-a_k)}{
1 + (x-a_k)^6 + (y-b_k)^6} 
\,, 
\qquad
\label{yode2t}
\end{eqnarray}
which is a planar ODE system with its right-hand side given as polynomials of degree $n(K)-1=6K-3$. Since we only rescaled the time, Figure~\ref{fig1}(a) provides an illustrative dynamics of the ODE system~(\ref{xode2t})--(\ref{yode2t}) for $K=4$. The illustrative trajectories have been calculated in Figure~\ref{fig1}(a) by solving ODEs~(\ref{xode2})--(\ref{yode2}) in time interval $t \in [0,100]$ and we can obtain the same result by solving ODEs~(\ref{xode2t})--(\ref{yode2t}) numerically in time interval $\tau \in [0, 10^{-9}]$. Applying $x$-factorable transformation to ODEs~(\ref{xode2t})--(\ref{yode2t}), we obtain
\begin{eqnarray}
\frac{\mbox{d}x}{\mbox{d}\tau} 
\!&=&\! 
x \, h(x,y) 
\sum_{k=1}^K 
\frac{(x-a_k)
\big\{1-(x-a_k)^2-(y-b_k)^2\big\}
-(y-b_k)}{
1 + (x-a_k)^6 + (y-b_k)^6} 
\,, 
\qquad
\label{xode2t2}\\
\frac{\mbox{d}y}{\mbox{d}\tau} 
\!&=&\! 
y \, h(x,y) 
\sum_{k=1}^K 
\frac{(y-b_k)
\big\{1-(x-a_k)^2-(y-b_k)^2\big\}
+(x-a_k)}{
1 + (x-a_k)^6 + (y-b_k)^6} 
\,, 
\qquad
\label{yode2t2}
\end{eqnarray}
which is a kinetic system of ODEs with polynomials of degree $n(K)=6K-2$ and which has $K$ stable limit cycles. Solving for $K$, we obtain $K=(n(K)+2)/6$, which establishes the lower bound~(\ref{Cnbound}) in Theorem~\ref{thmcrn3}.

\section{Discussion}
\label{sec9}

The main results of this paper have been formulated as Theorems~\ref{thmcrn1}, \ref{thmcrn2} and~\ref{thmcrn3}, which show that there exist CRNs with $K$ stable limit cycles for any integer $K \in {\mathbb N}.$ The CRN presented in our proof of Theorem~\ref{thmcrn1} consisted of $N(K)=K+2$ chemical species~$\mathcal{S}$ given by~(\ref{setSchemsp}) and $M(K) = 29 \, K$ chemical reactions~$\mathcal{R}$ (of at most seventh order) given by~(\ref{deduced_reactions}). The number of species and chemical reactions further increases in our proof of Theorem~\ref{thmcrn2}, where we restrict our investigation to CRNs with (at most) second-order kinetics. On the other hand, if we restrict to CRNs with only $N=2$ chemical species, then the order of the chemical reactions increases with $K$ as $n(K)=6K-2$ in our proof of Theorem~\ref{thmcrn3}.

An important question is whether we can further decrease $N(K)$ (the number of chemical species) and $M(K)$ (the number of chemical reactions) in Theorems~\ref{thmcrn1} and~\ref{thmcrn2} and still obtain a CRN with $K$ stable limit cycles. One possibility to decrease $M(K)$ is to use one chemical reaction to interpret multiple terms on the right-hand sides of ODEs~(\ref{xode5})--(\ref{viode5}). We have already done this in the reaction set
$\mathcal{R}_i^*$ given by~(\ref{setristar}) with the reaction
\begin{equation}
V_i + 2X + 2Y \mathop{\longrightarrow}^{1} V_i + X + Y,
\label{vi2x2y}
\end{equation}
which corresponds to terms of the form $-v_i x^2 y^2$ appearing in both equations~(\ref{xode5}) and (\ref{yode5}). Another way to construct a CRN with reactions modelling the two terms, $-v_i x^2 y^2$, in the reaction rate equations~(\ref{xode5})--(\ref{yode5}), is to use one chemical reaction per one term on the right-hand side. That is, the chemical reaction~(\ref{vi2x2y}) could be replaced by two chemical reactions
$$
V_i + 2X + 2Y \mathop{\longrightarrow}^{1} V_i + X + 2Y,
\qquad
\mbox{and}
\qquad
V_i + 2X + 2Y \mathop{\longrightarrow}^{1} V_i + 2X + Y
$$
without modifying the form of the reaction rate equations~(\ref{xode5})--(\ref{yode5}). In particular, if our aim is to decrease the number $M(K)$ of chemical reactions, we could consider to `merge' some other reactions, which have the same reactants. For example, reaction lists~(\ref{setri}) and~(\ref{setristar}) contain chemical reactions
$$
V_i + 3Y \mathop{\longrightarrow}^{k_{i,7}} 2V_i+3Y,
\qquad
V_i + 3Y \mathop{\longrightarrow}^{k_{i,17}} V_i + 4Y.
$$
If these chemical reactions had the same reaction rate constants~$k_{i,7}$ and~$k_{i,17}$, then we could replace them by one chemical reaction given by
$$
V_i + 3Y \mathop{\longrightarrow}^{k_{i,7}} 2 V_i + 4Y
$$
and we would obtain a CRN which has $28 \, K$ chemical reactions rather than $29 \, K$, which we use in Theorem~\ref{thmcrn1}.
Consequently, there is potential to decrease the size of the constructed CRN by a careful choice of our parameters or by modifying the right-hand sides of reaction rate equations~(\ref{xode5})--(\ref{viode5}). However, the focus of our paper was on the existence proofs and we leave the improvement of bounds on $N(K)$ and $M(K)$ to future work.

Another possible direction to investigate is to consider more detailed stochastic description of CRNs, written as continuous time discrete space Markov chains and simulated by the Gillespie algorithm~\cite{Erban:2020:SMR}. Such simulations would help us to investigate how our parameters $a_i,$ $b_i$, $i=1,2,\dots,K$, needs to be chosen that the system not only has the limit cycles of comparable size (as we visualized in Figure~\ref{fig3} in the ODE setting), but it also follows each of these limit cycles with a similar probability (comparable to $1/K$). This could also be achieved by using the noise-control algorithm~\cite{Plesa:2018:NCM} for designing CRNs. This algorithm structurally modifies a given CRN under mass-action kinetics, in such a way that (i) controllable state-dependent noise is introduced into the stochastic dynamics, while (ii) the reaction rate equations are preserved. In particular, it could be used to introduce additional chemical reactions (which do not change the ODE dynamics), but lead to controllable noise-induced switching between different limit cycles.

\vskip 10mm

\noindent
{\bf Acknowledgements.}
This work was supported by the Engineering and Physical Sciences Research Council, grant
number EP/V047469/1, awarded to Radek Erban. This work was also supported by the
National Science Foundation, grant number DMS-1620403 and a Visiting Research Fellowship
from Merton College, Oxford, awarded to Hye-Won Kang.

\setlength{\bibsep}{3pt plus 0.3ex}

\end{document}